\documentclass[oneside,a4paper,11pt,reqno]{amsart}
\usepackage{amssymb,amsmath,amsthm}
\usepackage{graphics,graphicx}

\newtheorem{hypo}{Hypothesis}

\newtheorem{prop}[hypo]{Proposition}

\newtheorem{thm}[hypo]{Theorem}

\newtheorem{lem}[hypo]{Lemma}

\newtheorem{defi}[hypo]{Definition}

\newtheorem{rqe}[hypo]{Remark}

\newtheorem{coro}[hypo]{Corollary}

\def\B{\mathcal{B}}
\def\C{\mathcal{C}}

\def\J{{J}}

\newcommand {\refeq}[1] {(\ref{#1})}

\title{On the class of caustics by reflection of planar curves}
\date\today

\author{Alfrederic Josse}
\address{Universit\'e de Brest,
UMR CNRS 6205, Laboratoire de Math\'ematique de Bretagne Atlantique,
6 avenue Le Gorgeu, 29238 Brest cedex, France}
\email{alfrederic.josse@univ-brest.fr}
\author{Fran\c{c}oise P\`ene}
\address{Universit\'e de Brest,
UMR CNRS 6205, Laboratoire de Math\'ematique de Bretagne Atlantique,
6 avenue Le Gorgeu, 29238 Brest cedex, France}
\email{francoise.pene@univ-brest.fr}

\keywords{caustic, class, polar, intersection number, pro-branch,
Pl\"ucker formula\\
Fran\c{c}oise P\`ene is supported by the french ANR project GEODE (ANR-10-JCJC-0108)}
\begin{document}

\begin{abstract}
Given any light position $ S\in\mathbb P^2$ and any algebraic curve $\mathcal C$
of $\mathbb P^2$ (with any kind of singularities), 
we consider the incident lines coming from $ S$
(i.e. the lines containing $ S$)
and their reflected lines after reflection on the mirror curve $\mathcal C$. 
The caustic by reflection $\Sigma_{ S}(\mathcal C)$ is the Zariski closure
of the envelope of these reflected 
lines. We introduce the notion of reflected polar curve and express
the class of $\Sigma_{ S}(\mathcal C)$ in terms of intersection numbers
of $\mathcal C$ with the reflected polar curve, thanks to a fundamental lemma 
established in \cite{fredsoaz1}. This approach enables us
to state an explicit formula for the class of $\Sigma_{ S}(\mathcal C)$
in every case in terms of intersection numbers of the
initial curve $\mathcal C$.
\end{abstract}

\maketitle
\section*{Introduction}
Let $\mathbf{V}$ be a three dimensional complex vector space
endowed with some fixed basis.
We consider a light point $\mathcal S[x_0:y_0:z_0]
\in\mathbb P^2:=\mathbb P(\mathbf{V})$ and 
a mirror given by an irreducible algebraic curve $\mathcal C=V(F)$ of 
$\mathbb P^2$, with $F\in Sym^d(\mathbf{V}^\vee)$
($F$ corresponds to a polynomial of degree $d$ in $\mathbb C[x,y,z]$).
We denote by $d^\vee$ the class of $\mathcal C$.
We consider the caustic by reflection $\Sigma_S(\mathcal C)$  of the mirror curve 
$\mathcal C$ with source point $ S$.
Recall that $\Sigma_ S(\mathcal C)$
is the Zariski closure of the 
envelope of the reflected
lines associated to the incident lines coming from $ S$ after reflection off $\mathcal C$.
When $S$ is not at infinity,
Quetelet and Dandelin \cite{Quetelet,Dandelin} proved that the caustic 
by reflection $\Sigma_{ S}(\mathcal C)$ 
is the evolute of the $ S$-centered homothety (with
ratio 2) of the pedal  of $\mathcal C$ from $ S$
(i.e. the evolute of the orthotomic of $\C$ with respect to $S$).
This decomposition has also been used
in a modern approach by \cite{BGG1,BGG2,BG} to study the source genericity 
(in the real case). In \cite{fredsoaz1} we stated formulas for the
degree of the caustic by reflection of planar algebraic curves.

In \cite{Chasles}, Chasles proved that the class of $\Sigma_ S(\mathcal C)$ is equal
to $2d^\vee+d$ for a generic $(\mathcal C, S)$. 
In \cite{Brocard-Lemoyne}, Brocard and Lemoyne gave (without any proof) a 
more general formula only when 
$ S$ is not at infinity. The Brocard and Lemoyne formula 
appears to be the direct composition of formulas got by 
Salmon and Cayley in \cite[p. 137, 154]{Salmon-Cayley} 
for some geometric characteristics of evolute and pedal curves.
The formula given by Brocard and Lemoyne is not satisfactory for the following reasons.
The results of Salmon and Cayley apply only to curves having
no singularities other than ordinary nodes and cusps \cite[p. 82]{Salmon-Cayley},
but the pedal of such a curve is not necessarily 
a curve satisfying the same properties. 
For example, the pedal curve of the rational cubic $V(y^2z-x^3)$ from $[4:0:1]$
is a quartic curve with a triple ordinary point.
Therefore it is not correct to compose directly the formulas got 
by Salmon and Cayley as Brocard and Lemoyne apparently did (see also Section \ref{append1}
for a counterexample of the Brocard and Lemoyne formula for the class of the caustic by reflection).

Let us mention some works on the evolute and on its generalization in higher dimension 
\cite{Fantechi,Trifogli,CataneseTrifogli}. In \cite{Fantechi}, Fantechi gave a necessary
and sufficient condition for the birationality of the evolute of a curve and 
studied the number and type of the singularities of the general evolute.
Let us insist on the fact that there exist irreducible algebraic curves (other than
lines and circles) for which the evolute map is not birational.
This study of evolute is generalized in higher dimension by Trifogli in \cite{Trifogli} and by
Catanese and Trifogli \cite{CataneseTrifogli}.

The aim of the present paper is to give a formula for the class (with multiplicity) 
of the caustic by
reflection for any algebraic curve $\mathcal C$ (\textbf{without any restriction neither on 
the singularity 
points nor on the flex points}) and {\bf for any light position} $ S$ 
(including the case when $ S$ is at infinity or when $ S$ is on the curve $\mathcal C$).

In Section \ref{sec1}, we define the reflected lines $\mathcal R_m$
at a generic $m\in\mathcal C$ and the (rational) ``reflected map''
$R_{\mathcal{C},S}:\mathbb P^2\rightarrow\mathbb P^2$
mapping a generic $m\in\mathcal C$ to the equation of $\mathcal R_m$.

In section \ref{sec2}, we define the caustic by reflection $\Sigma_{ S}(\mathcal C)$, we give conditions ensuring that 
$\Sigma_{ S}(\mathcal C)$ is an irreducible curve and we
prove that its class is the degree of the image of $\mathcal C$
by $R_{\mathcal{C},S}$.

In section \ref{sec3},
we give formulas for the class of caustics by reflection
valid for any $(\mathcal C, S)$. These formulas describe
precisely how the class of the caustic
depends on geometric invariants of $\mathcal C$ and also on the relative
positions of $ S$ and of the two cyclic points $ I, J$ with respect to $\mathcal C$. 
As a consequence of this result, we obtain the following formula for the class of $\Sigma_{ S}(\mathcal C)$ valid for
any $\mathcal C$ of degree $d\ge 2$ and for a generic source position $ S$:
$$class(\Sigma_{ S}(\mathcal C))=2d^\vee+d-\Omega(\mathcal C,\ell_\infty)-\mu_I(\C)-\mu_J(\C),$$
where $\Omega(\mathcal C,\ell_\infty)$ is the contact number of $\mathcal C$ with the line at infinity $\ell_\infty$
and with $\mu_I(\C)$ and $\mu_J(\C)$ are the multiplicities number of respectively $I$ and $J$ on $\mathcal C$.

In Section \ref{exemples}, our formulas are illustrated on two examples of curves
(the lemniscate of Bernoulli and the quintic considered in \cite{fredsoaz1}).

In section \ref{append1}, we compare our formula with the one given by Brocard and Lemoyne for a light position not at infinity.
We also give an explicit counter-example to their formula.

In Section \ref{proof}, we prove our main theorem. In a first time, we give a formula for the class of the caustic in terms of intersection numbers of $\mathcal C$ with a generic ``reflected polar'' at the base points of $R_{\mathcal{C},S}$. In a second time, we compute these intersection numbers in terms of the degree $d$ and of the class $d^\vee$ of $\mathcal C$
but also in terms of intersection numbers of $\mathcal C$ with
each line of the triangle $I J S$.

In appendix \ref{append2}, we prove a useful formula expressing 
the classical intersection number in terms of probranches.
\section{Reflected lines $\mathcal R_m$ and rational map $R_{\mathcal{C},S}$}\label{sec1}
Recall that we consider a light position $ S[x_{0}:y_{0}:z_{0}]\in {\mathbb P}^{2}$
and an irreducible algebraic (mirror) curve $\mathcal{C}=V(F)$ of $\mathbb{P}^{2}$
given by  a homogeneous polynomial $F\in Sym^d(\mathbf{V})$ with
$d\ge 2$. 
We write $Sing(\C)$ for the set of singular points of $\C$.
For any non singular point $m$, we write $\mathcal T_m\mathcal  C$ for the tangent line to $\mathcal C$ at $m$.
We set $\mathbf{S}(x_0,y_0,z_0)\in\mathbf{V}\setminus\{\mathbf{0}\}$.
For any $m[x:y:z]\in\mathbb P^2$, we write $\mathbf{m}(x,y,z)
\in\mathbf{V}\setminus\{\mathbf{0}\}$.
We write as usual $\ell_\infty=V(z)\subset\mathbb P^2$ for the line at infinity.
For any $\mathbf{P}(x_1,y_1,z_1)\in\mathbf{V}\setminus\{\mathbf{0}\}$, we define
$$\Delta_{\mathbf{P}}F:=x_1F_x+y_1F_y+z_1F_z\in Sym^{d-1}(\mathbf{V}^\vee).$$
Recall that $V(\Delta_{\mathbf P}F)$ is the polar curve of
$\mathcal C$ with respect to $P[x_1:y_1:z_1]\in\mathbb P^2$.

Since the initial problem is euclidean, we endow $\mathbb{P}^{2}$ with an angular
structure for which ${I }[1:i:0]\in\mathbb P^2$ and ${ J}[1:-i:0]\in\mathbb P^2$ 
play a particular role. To this end, let
us recall the definition of the cross-ratio $\beta$ of 4 points of $\ell_\infty$. 
Given four points
$(P_i[a_i:b_i:0])_{i=1,...,4}$ such that each point appears at most 2 times,
we define the cross-ratio $\beta(P_1,P_2,P_3,P_4)$ of these four points as follows~:
\begin{equation*}
\beta (P_{1},P_{2},P_{3},P_{4})=\frac{(b_3a_1-b_1a_3)(b_4a_2-b_2a_4)}{
    (b_3a_2-b_2a_3)(b_4a_1-b_1a_4)},
\end{equation*}
with convention $\frac{1}{0}=\infty $.
For any distinct lines $\mathcal{A}_1$ and $\mathcal{A}_2$ not equal to $\ell_\infty$, containing
neither $I$ nor $ J$, 
we define the oriented angular measure between $\mathcal{A}_1$ and $\mathcal{A}_2$ by $\theta$ (modulo $\pi \mathbb Z$)
such that 
$$e^{-2i\theta}=\beta(P_1,P_2,{ I},{ J})
=\frac{(a_1+ib_1)(a_2-ib_2)}
      {(a_1-ib_1)(a_2+ib_2)}$$
(where $P_i[a_i:b_i:0]$ is the point at infinity of $\mathcal{A}_i$).
Let $Q\in Sym^2(\mathbf{V}^\vee)$ be defined by
$Q(x,y):=x^2+y^2$. It will be worth noting that 
$Q(\nabla F)=F_x^2+F_y^2=\Delta_IF\Delta_JF$.
For every non singular point $m$ of $\mathcal C\setminus\ell_\infty$,
we recall that 
$t_{m}[F_{y}:-F_{x}:0]\in\mathbb P^2$ is the point at infinity of ${\mathcal T}_{m}\mathcal{C}$ and so $t_m\not\in\{I,J\}$
is equivalent to $m\not\in V(Q(\nabla F))$.

Now, for any $m\in\mathcal C\setminus(\ell_\infty\cup
Q(\nabla F))$ and  any incident line $\ell$ containing $m$,
we define as follows the associated reflected line 
$\mathfrak{R}_m(\ell)$ (for the reflexion on $\mathcal C$ at $m$
with respect to the Snell-Descartes reflection law
$Angle(\ell,\mathcal T_m)=Angle(\mathcal T_m,\mathfrak{R}_m)$).
\begin{defi}
For every
$m\in{\mathcal C}\setminus 
(\ell_\infty\cup V(Q(\nabla F)))$, we define $\mathfrak{r}_m:\ell_\infty\rightarrow\ell_\infty$ mapping $P\in\ell_\infty$
to the unique $\mathfrak{r}_m(P)$ such that $\beta(P,t_m,I,J)=\beta(t_m,\mathfrak{r}_m(P),I,J)$.

We define $\mathfrak{R}_m:\mathcal F_m\rightarrow\mathcal F_m$
with $\mathcal F_m:=\{\ell\in G(1,\mathbb P^2),\ \ m\in\ell\}$
by $\mathfrak{R}_m(\ell)=(m\, \mathfrak{r}_m(P_\ell))$ if $P_\ell$
is the point at infinity of $\ell$.
\end{defi}
We have (on coordinates)
$$ \mathfrak{r}_m([x_1:y_1:0])=
[x_1(F_x^2-F_y^2)+2y_1F_xF_y:-y_1(F_x^2-F_y^2)+2x_1F_xF_y:0]$$
\begin{rqe}
Observe that $\mathfrak{r}_m$ is an involution on $\ell_\infty\cong \mathbb P^1$
with exactly two fixed points $t_m$ and $n_m[F_x:F_y:0]$.
As a consequence,  
$\mathfrak{R}_m$ is an involution with two fixed points $\mathcal T_m(\mathcal C)$ and $\mathcal N_m(\mathcal C):=(mn_m)$ the normal line to $\mathcal C$ at $m$.

Moreover $\mathfrak{r}_m(I)=J$ and $\mathfrak{r}_m(J)=I$.
\end{rqe}
\begin{defi}\label{reflected}
For any $m[x:y:z]\in\mathcal C\setminus(\{ S\}\cup\ell_\infty\cup
V(Q(\nabla F))$ we define the {\bf reflected line} $\mathcal{R}_m$ on
$\mathcal C$ at $m$
(of the incident line coming from
$ S$) as the line $\mathcal R_m:=\mathfrak{R}_m((mS))$.
\end{defi}
For  $m[x:y:z]\in\mathcal C\setminus(\{ S\}\cup\ell_\infty\cup
V(Q(\nabla F))$, the point at infinity of $({ S}m)$
is $s_m[ x_0z-z_0x: y_0z-z_0y:0]$.
Due to the Euler identity, on $\mathcal C$, we have
$xF_x+yF_y+zF_z=0$ and so $(x_0z-z_0x)F_x+(y_0z-z_0y)F_y=z\Delta_{\mathbf S}F$. Hence 
$\mathfrak{r}(s_m)=[-v_{\mathbf{m}}:u_{\mathbf{m}}:0]$  and the reflected line
$\mathcal{R}_m$ is the set of $P[X:Y:Z]\in\mathbb P^2$
such that $u_{\mathbf m}X+v_{\mathbf m}Y+w_{\mathbf{m}}Z=0$, with 
\begin{eqnarray*}
u_{\mathbf{m}}&:=&(z_0y-zy_0)(F_x^2+F_y^2)+2z\Delta_{\mathbf{S}}F.F_y\in Sym^{2d-1}(\mathbf{V}^\vee)\\
v_{\mathbf{m}}&:=&(zx_0-z_0x)(F_x^2+F_y^2)-2z\Delta_{\mathbf{S}}F.F_x
\in Sym^{2d-1}(\mathbf{V}^\vee)\\
w_{\mathbf{m}}&:=&\frac{-xu_m-yv_m}z=(xy_0-yx_0)(F_x^2+F_y^2)-2\Delta_{\mathbf{S}}F(xF_y-yF_x)\in Sym^{2d-1}(\mathbf{V}^\vee).
\end{eqnarray*}
\begin{defi}
We call {\bf reflected map of $\mathcal C$ from $S$} the following rational map $$R_{\mathcal{C},S}:\begin{array}{ccc}\mathbb P^2&\rightarrow&
\mathbb P^2\\
m&\mapsto&[u_{\mathbf{m}}:v_{\mathbf{m}}:w_{\mathbf{m}}]
\end{array}.$$

We also define
the rational map $T_{\C,S}:=(R_{\mathcal{C},S})_{|\mathcal C}:\mathcal C\rightarrow \mathbb P^2$.
\end{defi}
For any $\mathbf{m}\in\mathbf{V}$, it will be useful
to define $\mathbf{R}_{F,\mathbf{S}}(\mathbf{m}):=(u_{\mathbf m},
v_{\mathbf m},w_{\mathbf m})\in\mathbf{V}$ and to notice that
$$\mathbf{R}_{F,\mathbf{S}}(\mathbf{m})=Q(\nabla F(\mathbf{m}))\cdot
(\mathbf{m}\wedge \mathbf{S})-2\Delta_{\mathbf S}F(\mathbf{m})\cdot
\left(\mathbf{m}\wedge\mathbf{n}_{\mathbf{m}}\right)\in\mathbf V, $$
with\footnote{with $\wedge:\mathbf V\times\mathbf V\rightarrow\mathbf V$ being given in coordinates by $(x_1,y_1,z_1)\wedge(x_2,y_2,z_2)= \left(\begin{array}{c}
z_2y_1-z_1y_2\\
z_1x_2-z_2x_1\\
x_1y_2-y_1x_2\\
\end{array}\right)$.}
 $\mathbf{n}_{\mathbf{m}}(F_x(\mathbf{m}),F_y(\mathbf{m}),0)\in\mathbf{V}$.
\begin{prop}
The base points of $T_{\C,S}$ are the following:

$ I$, $ J$, $ S$ (if these points are in $\mathcal C$), 
the singular points of $\mathcal C$ and the points of tangency
of $\mathcal C$ with some line  of the triangle $( IJS)$.
\end{prop}
\begin{proof}
We have to prove that the set of base points of $T_{\C,S}$ is the following set: $\mathcal M:=\C\cap(\{I,J,S\}\cup V(\Delta_{\mathbf S}F,
Q(\nabla F))\cup V(F_x,F_y))$. We just prove that $Base(T_{\C,S})
\subset\mathcal M$, the converse being obvious (observe that if $m\in\{I,J\}$, we automatically have $Q(\nabla F(\mathbf{m}))=0$ and
$\mathbf{n}_{\mathbf m}\in Vect(\mathbf{m})$).
Let $m[x;y;z]\in\C$ be such that $\mathbf{R}_{F,\mathbf{S}}(\mathbf{m})=
\mathbf{0}$. Then $\mathbf{m}$ and $Q(\nabla F(\mathbf{m}))\cdot
 \mathbf{S}-2\Delta_{\mathbf S}F(\mathbf{m})\cdot
\mathbf{n}_{\mathbf{m}}$ are colinear.
Due to the Euler identity, we have $0=DF(\mathbf{m})\cdot\mathbf{m}$
(with $DF(\mathbf{m})$ the differential of $F$ at $\mathbf{m}$)
and so $0=-\Delta_{\mathbf S}F(\mathbf{m})\cdot Q(\nabla F(\mathbf{m}))$
since $DF(\mathbf{m})\cdot \mathbf{S}=\Delta_{\mathbf S}F(\mathbf{m})$
and since $DF(\mathbf{m})\cdot\mathbf{n}_{\mathbf{m}}=Q(\nabla F(\mathbf{m}))$. Hence $\Delta_{\mathbf S}F(\mathbf{m})= 0$
or $Q(\nabla F(\mathbf{m}))=0$.

If $\Delta_{\mathbf S}F(\mathbf{m})= 0$, then either $Q(\nabla F(\mathbf{m}))=0$ or $m=S$.

If $\Delta_{\mathbf S}F(\mathbf{m})\ne 0$ and $Q(\nabla F(\mathbf{m}))=0$ , then $F_x(\mathbf{m})=F_y(\mathbf{m})=0$ or $m=[F_x(\mathbf{m}):F_y(\mathbf{m}):0]$. Assume that $m=[F_x(\mathbf{m}):F_y(\mathbf{m}):0]$.
Then, since $Q(\nabla F(\mathbf{m}))=0$, we conclude that $m\in\{I,J\}$.
\end{proof}
In the following result, we state the $S$-generic birationality
of $T_{\C,S}$. We give a short version of the proof of \cite{fredsoaz3}.
Let us indicate that another proof of the same result has been established at the same period by Catanese in \cite{Catanese}.
\begin{prop}[see also \cite{fredsoaz3,Catanese}]\label{birationalite}
Let $\mathcal C$ be an irreducible curve of degree $d\ge 2$.
Then, for a generic $ S\in\mathbb P^3$, the map
$T_{\C,S}$ is birational.
\end{prop}
\begin{proof}
For every $m\in\mathcal C_0:=\mathcal C\setminus(\ell_\infty\cup 
V(Q(\nabla F)))$ and every $ S\in\mathbb P^2\setminus\{m\}$,
we write $\mathcal R_{m, S}$ for the reflected line $\mathfrak{R}_m((mS))$.
For every $m\in\mathcal C_0$,
we consider the set
$K_m:=\{ S\in\mathbb P^2\setminus\mathcal C: \exists m'\in\mathcal C_0\setminus\{m\},\ \mathcal R_{m, S}=\mathcal R_{m', S}\}.$ 
\begin{itemize}
\item Let us prove that, for any $m\in\mathcal C_0$, $K_m$ is contained in an algebraic curve $\bar K_m$ of degree less than $2 d^2+2$.

Let $m\in\mathcal C_0$.
Consider $ S\in \mathbb P^2\setminus\mathcal C$ and $m'\in \mathcal C_0
\setminus\{m\}$ such that $\mathcal R_{m, S}=\mathcal R_{m', S}$. Then we have $\mathcal R_{m, S}=\mathcal R_{m', S}=(mm')$ and so $\mathcal S\in\mathcal R_{m,m'}\cap \mathcal R_{m',m}$.

Assume first that $\mathcal R_{m,m'}= \mathcal R_{m',m}$. Then 
this line is $(mm')$ and it is its own reflected line
both at $m$ and at $m'$. This implies that $(mm')$ is either $\mathcal T_m\mathcal C$ or $\mathcal N_m\mathcal C$, so that
$\mathcal S\in\mathcal T_m\mathcal C\cup \mathcal N_m\mathcal C$.

Assume now that $\mathcal R_{m,m'}\ne\mathcal R_{m',m}$.
Then $S=\tau_m(m')$ with
$\tau_m:\mathbb P^2\rightarrow\mathbb P^2$ 
the rational map associated to $\boldsymbol{\tau}_{\mathbf{m}}
:\mathbf{V}\rightarrow\mathbf{V}$
with
$\boldsymbol{\tau}_{\mathbf{m}}(\mathbf{m'})=\mathbf{R}_{F,\mathbf{m'}}(\mathbf{m})\wedge \mathbf{R}_{F,\mathbf{m}}(\mathbf{m'})$. 
Hence $K_m\subseteq \bar K_m:=
\mathcal T_m\mathcal C\cup \mathcal N_m\mathcal C\cup \overline{\tau_m(\mathcal C)}$,
where $\overline A$ is the Zariski closure of a set $A$.
Since the degree (in $m'$) of the coordinates of $\boldsymbol{\tau}_{\mathbf{m}}$
is $2d$, we conclude that $\deg \bar K_m\le 2d^2+2$.

\item The set $K$ of points $ S\in\mathbb P^2 \setminus\mathcal C$ 
such that $R_{\mathcal{C},S}$ is not birational is contained in
$$\bar K:=\bigcup_{E\subset \mathcal C_0:\# E<\infty}
  \bigcap_{m\in \mathcal C_0\setminus E}\bar K_m.$$
To conclude we will apply the Zorn lemma.
We have to prove that 
$\{\bigcap_{m\in \mathcal C_0\setminus E}\bar K_m,\ \# E<\infty\}$
is inductive for the inclusion. Let
$(\mathcal F_j:=\bigcap_{m\in \mathcal C_0\setminus E_j}\bar K_m)_{j\ge 1}$ be an increasing sequence of sets (with $E_j$ finite subsets of $\mathcal C_0$). Write $Z$ for the union of these sets.
Observe that $Z\subseteq \bar K_{m_0}$ for some fixed $m_0\in
\mathcal C_0\setminus\bigcup_{i\ge 1}E_i$. The set $\bar K_{m_0}$
is the union of irreducible algebraic curves $C_1,...,C_p$.
We write $d_i$ for the degree of $C_i$. If $C_i\subseteq Z$,
we write $N_i:=\min\{j\ge 1:C_i\subset\mathcal F_j\}$.
If $C_i\not\subseteq Z$, then $(C_i\cap\mathcal F_j)_{j\ge 1}$
is an increasing sequence of finite sets containing at most
$d_i(2d^2+2)$ points and we set $N_i:=\min\{j:(C_i\cap Z)\subseteq\mathcal F_j\}$. We obtain $Z=\mathcal F_{\max(N_1,...,N_p)}$.
Due to the Zorn lemma, there exists a finite set
$E_0$ such that 
$ K\subset\bigcap_{m\in \mathcal C_0\setminus E_0}\bar K_m$,
from which the result follows.
\end{itemize}
\end{proof}
\section{Caustic by reflection}\label{sec2}
\begin{defi}
The {\bf caustic by reflection} $\Sigma _{ S}(\mathcal{C})$ is the 
Zariski closure of the envelope of the
reflected lines $\{{\mathcal R}_{m};m\in \mathcal{C}\setminus
(\{S\}\cup\ell_\infty\cup V(Q(\nabla F))\}$.
\end{defi}
Recall that, in \cite{fredsoaz1}, we have defined a rational map
$\Phi_{F,\mathbf{S}}$ called \textbf{caustic map} mapping a generic $m\in\mathcal C$ to the point of tangency of $\Sigma_{S}
 (\mathcal C)$ with $\mathcal R_m$ and that $\Sigma_S(\C)$ is the Zariski closure of $\Phi_{F,\mathbf{S}}(\C)$.

In the present work, we will not consider the cases in which the caustic by reflection
$\Sigma_S(\mathcal C)$ is a single point.
We recall that these cases are easily characterized as follows.
\begin{prop}\label{nontrivial}
Assume that 
\begin{itemize}
\item[(i)] $S\not\in\{I,J\}$, 
\item[(ii)] $\C$ is not a line (i.e. $d\ne 1$),
\item[(iii)] if $d=2$, then $S$ is not a focus of the conic $\C$.
\end{itemize}
Then $\Sigma_S(\mathcal C)$ is not reduced to a point and is an irreducible curve.
\end{prop}
\begin{proof}
Assume (i), (ii) and (iii) and that $\Sigma_S(\C)=\{S'\}$
with $S'=[x_1:y_1:z_1]$.

When $S\not\in\ell_\infty$, we will use the fact that $\Sigma_S(\C)$ is the evolute 
of the orthotomic of $\C$ with respect to $S$.
Since $C$ is not a line, the orthotomic of $\C$ with respect to $S$ is not reduced to a point
but its evolute is a point.
This implies that the orthotomic of $\C$ with respect to $S$ is either a line (not equal
to $\ell_\infty$) or a circle.
But $\C$ is the contrapedal (or orthocaustic) curve (from $S$) 
of the image by the $S$-centered homothety (with ratio $1/2$)
of the orthotomic of $\C$.
Therefore $d=2$ and $S$ is a focal point of $\C$, which contradicts (iii).

When $S\in\ell_\infty$ but $S'\not\in\ell_\infty$, then, for symetry
reasons, we also have $\Sigma_{S'}(\C)=\{S\}$ and we conclude analogously.

Suppose now that $S,S'\in\ell_\infty$. We have $z_0=z_1=0$.
For every $m=[x:y:1]\in\C\setminus(\ell_\infty\cup V(Q(\nabla F)))$, we have
$\beta(S,t_m,I,J)=\beta(t_m,S',I,J)$
Therefore we have
$$\frac{(ix_0-y_0)(-iF_y+F_x)}
    {(iF_y+F_x)(-ix_0-y_0)}=\frac{(iF_y+F_x)(-ix_1-y_1)}
     {(-iF_y+F_x)(ix_1-y_1)} $$
and so
$${(ix_0-y_0)(ix_1-y_1)(-iF_y+F_x)^2}
    ={(iF_y+F_x)^2(-ix_0-y_0)(-ix_1-y_1)}.$$
Now, according to (i), $ix_0-y_0\ne 0$, $-ix_0-y_0\ne 0$, 
$ix_1-y_1\ne 0$, $-ix_1-y_1\ne 0$. Hence $(-iF_y+F_x)^2=a(iF_y+F_x)^2$
for some $a\ne 0$, which implies that $d=1$ and contradicts (ii).

Hence we proved that $\Sigma_S(\C)$ is not reduced to a point.
Now the irreducibility of $\Sigma_S(\C)$ comes from the fact that
$\Sigma_S(\C)=\overline{\Phi_{F,\mathbf{S}}(\C)}$ and that $\C$ is
an irreducible curve.
\end{proof}
\begin{prop}\label{prop1}
Assume that $\Sigma_{S}(\mathcal C)$ is not reduced to a point.
Then we have
\begin{equation}
\mbox{class}(\Sigma_S(\C))=\mbox{deg}(\overline{T_{\mathcal{C},S}(\C)}),
\end{equation}
where $\overline{T_{\mathcal{C},\mathbf{S}}(\C)}$ stands for the Zariski closure of $T_{\mathcal{C},\mathbf{S}}(\C)$.
\end{prop}
\begin{proof}
This comes from the fact that $\Sigma_S(\C)$ is the Zariski closure of the envelope of 
$\{\mathcal R_m,\ m\in \C\setminus(Sing(\C)\cup\{S\}\cup\ell_\infty\cup V(Q(\nabla F))\}$ and
can be precised as follows. 
For every algebraic curve $\Gamma =V(G)$ (with  $G$
in $Sym^k(\mathbf{V}^\vee)$ for some $k$), we consider
the Gauss map $\delta
_{\Gamma}:{\mathbb{P}}^{2}\longrightarrow {\mathbb{P}}^{2}$
defined on coordinates by $\delta _{\Gamma}([x:y:z])=[G_x:G_y:G_z]$, we
obtain immediately  
the commutative diagram :
\begin{equation}\label{diagramme}
\begin{array}{ccc}
 \mathcal{C} & \overset{\left(\Phi _{F,\mathbf{S}}\right)_{|\mathcal{C}} } {\longrightarrow }
& \Sigma _{ S}(\mathcal{C}) \\ 
  & \overset{T_{\mathcal{C},{S}}}{\searrow } &  \downarrow \delta _{\Sigma _{S}  (\mathcal C)} \\ 
  &  & \delta _{\Sigma _{S}
  (\mathcal C)}(\Sigma _{S} (\mathcal C))
   \cong(\Sigma _{S} (\mathcal C))^{\vee } \end{array},
\end{equation}
with $\Phi_{F,\mathbf{S}}$ the caustic map defined in \cite{fredsoaz1}
(see the begining of the present section).
\end{proof}
Let us notice that, according to the proof of Proposition \ref{prop1},
the rational map $T_{\mathcal{C},S}$ as the same degree as the rational map
$(\Phi_{F,\mathbf{S}})_{|\mathcal C}$ (since $\Sigma_S(\C)$ is irreducible and since the
Gauss map $(\delta_{\Sigma_{S}(\mathcal C)})_{|\Sigma_{S}(\mathcal C)}$ 
is birational \cite{Fisher}).
\section{Formulas for the class of the caustic}\label{sec3}
Since the map $T_{\mathcal{C},S}$ 
may be non birational, we introduce
the notion of class with multiplicity 
of $\Sigma_S(\C)$:
$${\textrm{mclass}} (\Sigma_{S}(\mathcal C))=\delta_1(S,\C)\times
{\textrm{class}} (\Sigma_{S}(\mathcal C))$$
where ${\textrm{class}} (\Sigma_{S}(\mathcal C))$ is the class of 
the algebraic curve $\Sigma_S(\C)$
and where $\delta_1(S,\C)$ is the degree of the rational map $T_{\mathcal{C},{S}}$.
We recall that $\delta_1(S,\C)$ corresponds to the number of preimages
on $\mathcal C$ of a generic point of $\Sigma_S(\C)$ by $T_{\mathcal{C},{S}}$.

Before stating our main result, let us introduce some notations.
For every $m_1\in\mathbb P^2$, we write $\mu_{m_1}=
  \mu_{m_1}(\mathcal C)$ for the multiplicity of $m_1$
on $\mathcal C$ and consider the set
$Branch_{m_1}(\mathcal C)$ of branches of $\mathcal C$ at $m_1$.
We denote by $\mathcal E$ the set of couples point-branch $(m_1,\mathcal B)$ of $\mathcal C$
with $m_1\in\mathcal C$ and $\mathcal B\in Branch_{m_1}(\mathcal C)$.
For every $(m_1,\mathcal B)\in\mathcal E$, we write $e_{\mathcal B}$ for the multiplicity
of $\mathcal B$ and $\mathcal T_{m_1}(\mathcal B)$ the tangent line to $\mathcal B$
at $m_1$; we observe that $\mu_{m_1}=\sum_{\mathcal B\in Branch_{m_1}(\mathcal C)}
e_{\mathcal B}$. 
We write $i_{m_1}(\Gamma,\Gamma')$ the intersection number of
two curves $\Gamma$ and $\Gamma'$ at $m_1$.
For any algebraic curve $\mathcal C'$ of $\mathcal P^2$, 
we also define the contact number $\Omega_{m_1}(\mathcal C,\mathcal C')$ 
of $\mathcal C$ and $\mathcal C'$ at $m_1\in{\mathbb P}^2$ by
$$\Omega_{m_1}(\mathcal C,\mathcal C'):=i_{m_1}(\mathcal C,\mathcal C')-\mu_{m_1}(\mathcal C)
\mu_{m_1}(\mathcal C')\ \ \ \ \mbox{if}\ m_1\in\C\cap\C'$$
and
$$\Omega_{m_1}(\mathcal C,\mathcal C'):=0\ \ \ \ 
\mbox{if}\ m_1\not\in\C\cap\C'.$$
Recall that
$\Omega_{m_1}(\mathcal C,\mathcal C')=0$ means that $m_1\not\in\mathcal C\cap\mathcal C'$ or that
$\mathcal C$ and $\C'$ intersect transversally at $m_1$.
\begin{thm}\label{THM1}
Assume that the hypotheses of Proposition \ref{nontrivial} hold true.
\begin{enumerate}
\item If $S\not\in\ell_\infty$, the class (with multiplicity) 
of $\Sigma_ S(\mathcal C)$ is given by
\begin{equation}\label{classfini}
\mbox{mclass}(\Sigma_{ S}(\C))
=2d^\vee+d-2  f'-  g-  f-  g'+  q',
\end{equation}
where
\begin{itemize}
\item $g$ is the contact number of $\mathcal C$ with $\ell_\infty$, i.e.
$  g:=\sum_{m_1\in\C\cap\ell_\infty}\Omega_{m_1}(\mathcal C,\ell_\infty), $
\item $  f$ is the multiplicity number at a cyclic point of $\mathcal C$ with an isotropic line from 
$S$, i.e.
$$  f:=i_{I}(\mathcal C,(I S))+
    i_{ J}(\mathcal C,(J S)),$$ 
\item $  f'$ is the contact number of $\mathcal C$ with an isotropic line from $S$
outside $\{I, J, S\}$, i.e.
$$  f':= \sum_{m_1\in(\C\cap(IS))\setminus \{I, S\} }
  \Omega_{m_1}(\mathcal C,( I S))+
  \sum_{m_1\in(\C\cap(JS))\setminus \{ J, S\} }
  \Omega_{m_1}(\mathcal C,( J S)),$$
\item $  g'$ given by
$  g':=i_{S}(\mathcal C,( IS))+i_{S}(\mathcal C,( JS))
-\mu_{ S};$
\item $ q'$ is given by
$$  q':= \sum_{(m_1,\mathcal B)\in\mathcal E:m_1\not\in\{ I, J, S\},
T_{m_1}\mathcal B=( I S)\ or\ T_{m_1}\mathcal B=( J S),\
    i_{m_1}(\mathcal B,\mathcal T_{m_1}(\mathcal B))\ge 2e_{\mathcal B}}
[i_{m_1}(\mathcal B,\mathcal T_{m_1}(\mathcal B))-2e_\mathcal B].$$
\end{itemize}
\item 
If $ S\in\ell_\infty$, the class of $\Sigma_S(\mathcal C)$
is 
\begin{equation}\label{classinfini}
\mbox{mclass}(\Sigma_{S}(\C))
=2d^\vee+d- 2  g-\mu_{I}-\mu_{ J}
-\mu_{S}-c'(S),
\end{equation}
with
$$c'(S):=\sum_{\mathcal B\in Branch_{S}(\mathcal C):
i_{S}(\mathcal B,\ell_\infty)= 2 e_{\mathcal B}} (e_{\mathcal B}+\min(i_{S}(\mathcal B,Osc_{S}(\mathcal B))-3e_{\mathcal B},0)), $$
where $Osc_{S}(\mathcal B)$ is any smooth algebraic 
osculating curve to $\mathcal B$ at $S$ (i.e. 
any smooth algebraic curve $\mathcal C'$ such that $i_{S}(\mathcal B,\C')> 2e_{\B}$).
\end{enumerate}
\end{thm}
The notations introduced in this theorem
are directly inspired by those of Salmon and Cayley \cite{Salmon-Cayley} (see Section \ref{append1}).
Let us point out that, in this article, $g$ is not the geometric genus of the curve.
\begin{rqe}
Observe that we also have
$$c'(S):=\sum_{\mathcal B\in Branch_{S}(\mathcal C):
i_{S}(\mathcal B,\ell_\infty)= 2 e_{\mathcal B}} (e_{\mathcal B}+\min(
\beta_1(S,\mathcal B)-3e_{\mathcal B},0)),$$
where $\beta_1(S,\mathcal B)$ is the first characteristic exponent
of $\mathcal B$ non multiple of $e_\mathcal B$ (see \cite{Zariski}).

Observe that, when $i_{S}(\mathcal B,\mathcal T_{S}(\mathcal B))=2e_{\mathcal B}$, 
we have $\min(i_{S}(\mathcal B, Osc_{S}(\mathcal B))-3e_{\mathcal B},0)=0$ except
if $S$ is a singular point and if the probranches of $\mathcal B$ are given
by $Y-x_0^{-1}y_0=\alpha Z^2+\alpha_1 Z^{\beta_1}+...$ in the chart $X=1$ if $x_0\ne 0$
(or $X-y_0^{-1}x_0=\alpha Z^2+\alpha_1 Z^{\beta_1}+...$ in the chart $Y=1$ otherwise),
with $\alpha\ne 0$, $\alpha_1\ne 0$ and $2<\beta_1<3$. 
Hence $c'(S)=\sum_{\mathcal B\in Branch_{S}(\mathcal C):
i_{S}(\mathcal B,\ell_\infty)= 2 e_{\mathcal B}} e_{\mathcal B}$
when $\mathcal C$ admits no such branch tangent at $S$ to $\ell_\infty$.
\end{rqe}
Combining Proposition \ref{birationalite} and Theorem \ref{THM1}, we obtain
\begin{coro}[A source-generic formula for the class]
Let $\mathcal C\subset\mathbb P^2$  be
a fixed curve of degree $d\ge 2$.
For a generic
source point $S$, we have $\delta_1(S,\mathcal C)=1$
and $class(\Sigma_{S}(\mathcal C))=2d^\vee+d-g-\mu_I-\mu_J$
with $g$ the contact number of $\mathcal C$ with $\ell_\infty$.
\end{coro}
\begin{proof}
Due to Proposition \ref{birationalite}, $\delta_1( S,\mathcal C)=1$
for a generic $S\in\mathbb P^2$. So $class(\Sigma_{S}
(\mathcal C))=mclass(\Sigma_{S}(\mathcal C))$.

Assume moreover, that $S\not\in\ell_\infty$
(so we apply the first formula of Theorem \ref{THM1}), $S\not\in \C$ (so $g'=0$), that $( IS)$ and
$( JS)$ are not tangent to
$\mathcal C$ (so $f'=q'=0$ and $f=\mu_{ I}(\mathcal C)+\mu_{ J}(\mathcal C)$). We obtain the result.
\end{proof}
\section{Examples}\label{exemples}
Let us now illustrate our result for two particular mirror curves. 
\subsection{Example of the lemniscate of Bernoulli}
We consider the case when $\mathcal C=V(F)$ is the lemniscate of Bernoulli given
by $F(x,y,z)=(x^2+y^2)^2-2(x^2-y^2)z^2$ and when $S\in\mathbb P^2\setminus\{I,J\}$. 
The degree of $\mathcal C$ is $d=4$.
The singular points of $\mathcal C$ are : $ I[1:i:0]$, $ J[1:-i:0]$ and $\mathcal O[0:0:1]$.
These three points are double points, each one having two different tangent lines. 
Hence the class of $\mathcal C$ is given by $d^\vee=d(d-1)-3\times 2=6$ and so
$$2d^\vee+d=16. $$
The tangent lines to $\mathcal C$ at $ I$ are $\ell_{1, I}:=V(y-iz-ix)$ and 
$\ell_{2, I}:=V(y-iz+ix)$
(the intersection number of $\mathcal C$ with $\ell_{1, I}$ or with 
$\ell_{2, I}$ at $ I$ is equal to 4).
The tangent lines to $\mathcal C$ at $ J$ are $\ell_{1, J}:=V(y+iz-ix)$ and 
$\ell_{2, J}:=V(y+iz+ix)$ (the intersection number of $\mathcal C$ with $\ell_{1, J}$ or with 
$\ell_{2, J}$ at $ J$ is equal to 4).
This ensures that we have
$$
  f=2(2+{\mathbf 1}_{S\in\ell_{1, I}}+{\mathbf 1}_{S\in\ell_{2, I}}
      +{\mathbf 1}_{S\in\ell_{1, J}}+{\mathbf 1}_{S\in\ell_{2, J}} ).
$$
Observe that $\ell_\infty$ is not tangent to $\mathcal C$. Indeed $I$ and $J$ are the only points
in $\mathcal C\cap\ell_\infty$ and $\ell_\infty$ is not tangent to $\mathcal C$ at these points. Therefore we have
$  g=0$ and $c'(S)=0$.

Since $ I$ and $ J$ are also the only points  at which $\mathcal C$ is tangent to an isotropic line
(i.e. a line containing $ I$ or $J$), we have $  f'=0$, $  g'=\mu_{S}$, $  q'=0$.
In this case, one can check that $\delta_1(S,\mathcal C)=1$.
Finally, we get
\begin{equation}
\mbox{if}\ S\not\in\ell_\infty,\ \ \ 
\mbox{class}(\Sigma_{S}(\mathcal C))=12-2({\mathbf 1}_{S\in
  \ell_{1,I}\cup\ell_{2, I} }+{\mathbf 1}_{S\in
  \ell_{1, J}\cup\ell_{2, J} })-\mu_{S}.
\end{equation}
Moreover,
since $\mu_{ I}=\mu_{J}= 2$, we have
\begin{equation}
\mbox{if}\  S\in\ell_\infty\setminus\{ I, J\},\ \ \ 
\mbox{class}(\Sigma_{S}(\mathcal C))=16-2-2=12,
\end{equation}
(since $\mu_{ I}=\mu_{ J}= 2$ and since $\mu_{ S}=0$).
For example, for $S[1:0:1]$, we get $\mbox{class}(\Sigma_{S}(\mathcal C))=8$, since
$S$ is in $\ell_{2, I}\cap\ell_{1,\mathcal J}$ but
not in $\mathcal C$ (so $\mu_{S}=0$).
\subsection{Example of a quintic curve}
As in \cite{fredsoaz1}, we consider the quintic curve $\mathcal C=V(F)$ with $F(x,y,z)=y^2z^3-x^5$.
We also consider a light point $S[x_0:y_0:z_0]
   \in\mathbb P^2\setminus\{ I,J\}$.
This curve admits two singular points: $A_1[0:0:1]$ and $A_2[0:1:0]$, we have $d=5$.

We recall that $\mathcal C$  admits a single branch at $A_1$, which has multiplicity 2 and
which is tangent to $V(y)$. We observe that $i_{A_1}(\mathcal C,V(y))=5$.

Analogously, $\mathcal C$  admits a single branch at $A_2$, which has multiplicity 3 and
which is tangent to $\ell_\infty$. We observe that $i_{A_2}(\mathcal C,\ell_\infty)=5$.

We obtain that the class of $\mathcal C$ is $d^\vee=5$ and that $\mathcal C$
has no inflexion point (these two facts are proved in \cite{fredsoaz1}). In particular, we get that $2d^\vee+d=15$.

Since $A_2$ is the only point of $\mathcal C\cap\ell_\infty$, we get that 
$  g=\Omega_{A_2}(\mathcal C,\ell_\infty)=2$ and $  f=0$.

The curve $\mathcal C$ admits six (pairwise distinct) isotropic tangent lines other than $\ell_\infty$: 
$\ell_{1}$, $\ell_{2}$ and $\ell_{3}$ containing $ I$ 
$$\forall k\in\{1,2,3\},\ \ 
   \ell_{k}=V\left(ix-y+\frac{3i}{25}\alpha^k\sqrt[3]{20}z\right) ,\ \ \mbox{with}\ \alpha:=e^{\frac {2i\pi}3}$$
and $\ell_{4}$, $\ell_{5}$ and $\ell_{6}$ containing $\mathcal J$:
$$\forall k\in\{1,2,3\},\ \ \ell_{3+k}=V\left(ix+y+\frac{3i}{25}\alpha^k\sqrt[3]{20}z\right).$$
For every $i\in\{1,2,3,4,5,6\}$,
we write $a_i$ the point at which $\mathcal C$ is tangent to $\ell_i$ (the points $a_i$ correspond to the points
of $\mathcal C\cap V(F_x^2+F_y^2)\setminus\{A_1,A_2\}$).
Since $\mathcal C$ contains no inflexion point and since
$A_1$ and $A_2$ are the only singular points of $\mathcal C$, we get that, 
$$  f'=\#\{i\in\{1,2,3,4,5,6\}\ :\ S\in\ell_i\setminus\{a_i\}\}\ \ \mbox{and}\ \   q'=0$$
when $S\not\in\ell_\infty$.

Now recall that  $  g'=i_{S}(\mathcal C,( IS))+i_{ S}(\mathcal C,(\mathcal JS))
-\mu_{ S}$. 
Again, in this case, one can check that $\delta_1( S,\mathcal C)=1$.
If $ S\not\in\ell_\infty$, we have
\begin{equation}
\textrm{class}(\Sigma_{S}(\mathcal C))= 
   13 -2\times\#\{i\in\{1,2,3,4,5,6\}\ :\  S\in\ell_i\setminus\{a_i\}\}-  g'
\end{equation}
and if $ S\in\ell_\infty\setminus\{ I,\mathcal J\}$, we have
\begin{equation}
\textrm{class}(\Sigma_{ S}(\mathcal C))=11-3\times {\mathbf 1}_{ S=A_2}.
\end{equation}
We observe that the points of $\mathbb P^2\setminus\{ I,\mathcal J\}$ belonging to two dictinct 
$\ell_k$ are outside $\mathcal C$. The set of these points is
$$\mathcal E:=\bigcup_{k=1}^3\left\{\left[-\frac 3{25}\sqrt[3]{20}\alpha^k:0:1\right],\ 
  \left[\frac{3}{50}\sqrt[3]{20}\alpha^k:\frac 3{50}\sqrt{3}\sqrt[3]{20}\alpha^k:1\right],\ 
 \left[\frac{3}{50}\sqrt[3]{20}\alpha^k:-\frac 3{50}\sqrt{3}\sqrt[3]{20}\alpha^k:1\right]\right\}$$
with $\alpha=e^{\frac {2i\pi}3}$.
Finally, the class of the caustic in the different cases is summarized in the following table.
$$
\begin{array}{|c|c|}
\hline
\mbox{Condition on } S\in\mathbb P^2\setminus\{ I,\mathcal J\}
    &\textrm{class}(\Sigma_{ S}(\mathcal C))=\\
\hline
 S=A_2&8\\
\hline
 S\in\mathcal E&9\\
\hline
 S\in\mathcal C\cap \bigcup_{k=1}^6(\ell_k\setminus\{a_k\})&10\\
\hline
 S\in(\ell_\infty\setminus\{A_2\})\cup\left(\bigcup_{k=1}^6\ell_k\setminus(\mathcal E\cup \mathcal C)\right)
   \cup\{A_1\}\cup\{a_1,...,a_6\}&11\\
\hline
S\in \mathcal C\setminus\left(\ell_\infty\cup\{A_1\}\cup\bigcup_{k=1}^6\ell_k\right)&12\\
\hline
\mbox{otherwise}&13\\
\hline
\end{array}
$$

\section{On the formulas by Brocard and Lemoyne and by Salmon and Cayley}\label{append1}
\subsection{Formulas given by Brocard and Lemoyne}
Recall that, when $S\not\in\ell_\infty$, $\Sigma_S(\C)$ is the evolute of an homothetic
of the pedal of $\C$ from $S$.

The work of Salmon and Cayley is under ordinary Pl\"ucker conditions 
(no hyper-flex, no singularities other than ordinary cups and ordinary nodes).
In \cite[p.137]{Salmon-Cayley}, Salmon and Cayley gave the following formula for the class of
the evolute~:
$$n'=m+n-f-g. $$
Replace now $m$, $n$, $f$ and $g$ by $M$, $N$, $F$ and $G$ (respectively) given in
\cite[p. 154]{Salmon-Cayley} for the pedal. Doing so, one exactly get (with the same notations) 
the formula of the class of caustics by reflection given by 
Brocard and Lemoyne in \cite[p. 114]{Brocard-Lemoyne}.

As explained in introduction, this composition of formulas of Salmon and Cayley
is incorrect because of 
the non-conservation of the Pl\"ucker conditions by the pedal transformation.
Nevertheless, for completeness sake, let us present the Brocard and Lemoyne formula
and compare it with our formula.
Brocard and Lemoyne gave the following formula
for the class of the caustic by reflection $\Sigma_{S}(\mathcal C)$ when 
$S\not\in\ell_\infty$:
\begin{equation}\label{BL}
class(\Sigma_{ S}(\mathcal C)) = d+2(d^\vee-\hat f')-\hat g-\hat f-\hat g'+\hat q',
\end{equation}
for an algebraic curve $\mathcal C$ of degree $d$, of class $d^\vee$,
$\hat g$ times tangent to $\ell_\infty$,
passing $\hat f$ times through a cyclic point, $\hat f'$ times tangent to an isotropic line
of $ S$, passing $\hat g'$ times through $ S$, 
$\hat q'$ being the coincidence number of contact
points when an isotropic line is multiply tangent. In \cite{Salmon-Cayley}, $\hat q'$
is defined as the coincidence number of tangents at points $\iota_1$, $\iota_2$ 
of ${\mathbb P^2}^\vee$
(corresponding to $( I S)$ and
$( J S)$) if these points are multiple points 
of the image of  $\mathcal C$ 
by the polar reciprocal transformation with center $ S$; i.e. 
$\hat q'$ represents the number of ordinary flexes of $\C$.

When $ S\not\in\ell_\infty$, let us compare terms appearing in our formula
(\ref{classfini}) with terms of (\ref{BL}) :
\begin{itemize}
\item $\hat g$ seems to be equal to $   g$;
\item it seems that $\hat f=\mu_I+\mu_J$ and so 
$$   f
     =\hat f+\Omega_I(\C,(IS))+\Omega_J(\C,(JS));$$
\item it seems that $\hat f'=\sum_{m_1\in\C\cap(IS)}\Omega_{m_1}(\C,(IS))+
    \sum_{m_1\in\C\cap(JS)}\Omega_{m_1}(\C,(JS))$
and so
$$   f':=\hat  f'-\Omega_I(\C,(IS))-\Omega_\J(\C,(JS))-\Omega_S(\C,(IS))
       -\Omega_S(\C,(JS));$$
\item it seems that $\hat g'=\mu_S$, therefore
$$   g':=\hat g' +\Omega_{S}(\C,(IS))+\Omega_{S}(\C,(JS));$$
\item our definition of $   q'$ appears as an extension of $\hat q'$
(except that we exclude the points $m_1\in\{I,J,S\}$). 
\end{itemize}
Observe that these terms coincide with the definition of Brocard and Lemoyne
if  $( IS)$
and $( J S)$ are not tangent to 
$\mathcal C$ at $ S$, $ I$, $ J$. In particular, if we call $BL$ the
right hand side of \refeq{BL}, the first item Theorem \ref{THM1} states that, when
$S$ is not at infinity we have
$$mclass(\Sigma_S(\C))=BL +\Omega_I(\C,(IS))+\Omega_\J(\C,(JS))+\Omega_S(\C,(IS))+\Omega_S(\C,(JS)).$$
\subsection{A counterexample to the formula of Brocard and Lemoyne}
We consider an example in which $\Omega_I(\C,(IS))=\Omega_J(\C,(JS))=1$, which means that $(IS)$ is tangent to $\C$ at $I$
and $(JS)$ is tangent to $\C$ at $J$.
Let us consider the non-singular 
quartic curve $\mathcal C=V(2yz^3+2z^2y^2+2zy^3+2y^4-2z^3x+2zyx^2+5y^2x^2+3x^4)$
and $ S[0:0:1]$.
This curve $\mathcal C$ has degree $d=4$ and class $d^\vee=4\times 3=12$, 
is not tangent to $\ell_\infty$,
is tangent to $( S I)$ at $ I$ and nowhere else,
is tangent to $( S J)$ at $ J$ and nowhere else; these tangent
points are ordinary. $S$ is a non singular point of $\C$.
Therefore, with our definitions, we have
$g=0$, $f=2+2=4$, $f'=0$, $g'=1+1-1=1$, $q'=0$, 
which gives
$\mbox{class}(\Sigma_{ S}(\mathcal C))=4+2(12-0)-0-4-1-0=23$,
since in this case $\delta_1( S,\mathcal C)=1$.
In comparison, the Brocard and Lemoyne formula would give
$\hat g=0$, $\hat f=1+1=2$, $\hat f'=1+1=2$, $\hat g'=1$, $\hat q'=0$ and so their formula gives
$\mbox{class}(\Sigma_{ S}(\mathcal C))=4+2(12-2)-0-2-1-0=21$
but this is false!

\section{Proof of Theorem \ref{THM1}}\label{proof}
To compute the degree of $\overline{T_{\mathcal{C},{S}}(\mathcal C)}$, we will use the
Fundamental Lemma given in \cite{fredsoaz1}.
Let us first recall the definition of $\varphi$-polar introduced in \cite{fredsoaz1}
and extending the notion of polar.
\begin{defi}
Let $p\ge 1$, $q\ge 1$ and let $\mathbf{W}$ be a complex
vector space of dimension $p+1$. Given $\varphi :{\mathbb{P}}^{p}:=\mathbb P(\mathbf{W})\rightarrow {\mathbb{
P}}^q$ a rational map defined by $\varphi=[\varphi_0:\cdots:\varphi_q]$
(with $\varphi_1,\dots,\varphi_q\in Sym^d(\mathbf{W}^\vee)$)
and $a=[a_0:\cdots:a_q]\in {\mathbb{P}}
^q$, we define the $\varphi$-polar at $a$, denoted by ${\mathcal{P}}
_{\varphi,a}$, the hypersurface of degree $d$ given by 
$
{\mathcal{P}}_{\varphi,a}:=V\left( \sum_{j=0}^q a_j\varphi_j \right)\subseteq \mathbb P^p.
$
\end{defi}
With this definition, the ``classical'' polar of a curve ${\mathcal{C}}=V(F)$ of 
${\mathbb{P}}^{2}$ (for some homogeneous polynomial $F\in\mathbb C[x,y,z]$) at $a$ is the $
\delta_{\mathcal C}$-polar curve at $a$, where $\delta_{\mathcal C}:[x:y:z]\mapsto[F_x:F_y:F_z]$.
\begin{defi}
We call {\bf reflected polar (or $r$-polar) 
of the plane curve $\mathcal C$ with respect to
$S$ at $a$} the $R_{\mathcal{C},{S}}$-polar at $a$, i.e. the curve ${\mathcal P}^{(r)}
   _{S,a}(\mathcal C):={\mathcal{P}}_{R_{\mathcal{C}, S},a}$.
\end{defi}
From a geometric point of view, ${\mathcal P}^{(r)}
   _{S,a}(\mathcal C)$ is an 
algebraic curve such that, for every $m\in\mathcal C \cap {\mathcal P}^{(r)}
   _{S,a}(\mathcal C)$, $\mathcal{R}_m$ contains $a$ (if $\mathcal{R}_m$ is well defined),
this means that line $(am)$ is tangent to $\Sigma_{ S}(\mathcal C)$
at the point $m'=\Phi_{F,\mathbf{S}}(m)\in \Sigma_{ S}(\mathcal C)$ associated to $m$ (see
picture).

\begin{center}
\includegraphics[scale=0.43]{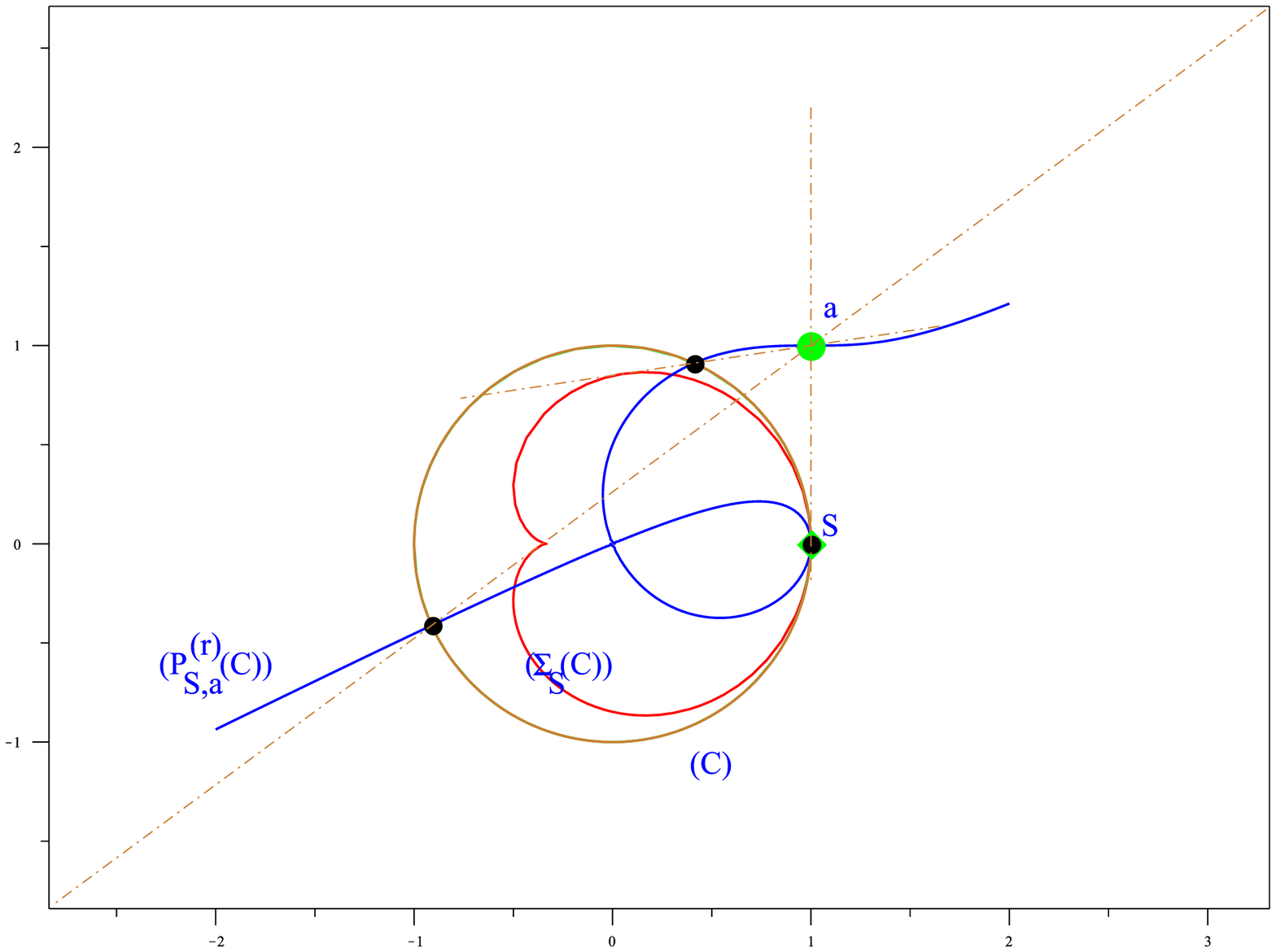}
\end{center}
%

Let us now recall the statement of the fundamental lemma proved in \cite{fredsoaz1}.

\begin{lem}[Fundamental lemma \cite{fredsoaz1}]\label{lemmefondamental}
Let $\mathbf{W}$ be a complex vector space of dimension
$p+1$, let
$\mathcal C$ be an irreducible algebraic curve of ${\mathbb P}^p
:=\mathbb P(\mathbf{W})$ and 
$\varphi :{\mathbb P}^p\rightarrow{\mathbb P}^q$ be a rational map 
given by $\varphi=[\varphi_0:\cdots:\varphi_q]$ with 
$\varphi_0,...,\varphi_q\in Sym^{\delta}(\mathbf{W}^\vee)$.
Assume that $\mathcal C\not\subseteq Base(\varphi)$ and that
$\varphi_{|\mathcal C}$ has degree $\delta_1\in \mathbb N\cup\{\infty\}$.
Then, for generic $a=[a_0:\cdots:a_q]\in{\mathbb P}^q$, the following 
formula holds true
$$\delta_1.\mbox{deg}\left(\overline{\varphi({\mathcal C})}\right)=\delta .\mbox{deg}
({\mathcal C})
     -\sum_{p\in Base(\varphi_{\vert \mathcal C})} i_p \left({\mathcal C},
{\mathcal P}_{\varphi,a}\right),$$
with convention $0.\infty=0$ and 
$\textrm{deg}(\overline{\varphi({\mathcal C})})=0$ if $\#\overline{\varphi({\mathcal C})}<\infty$.
\end{lem}
Due to this lemma and to Proposition \ref{prop1}, we have
\begin{equation}\label{formule1}
\mbox{mclass}(\Sigma_ S(\mathcal C))
=
d(2d-1)-\sum_{m_1\in Base(T_{\mathcal{C},{S}})}i_{m_1}(\mathcal C, {\mathcal P}^{(r)}
   _{ S,a}(\mathcal C)) .
\end{equation}
Now, we enter in the most technical stuff which is the computation of the intersection
numbers $i_{m_1}(\mathcal C, {\mathcal P}^{(r)}_{ S,a}(\mathcal C))$ of $\mathcal C$
with its reflected polar at the base points of $R_{\mathcal{C},S}$.
To compute these intersection numbers, it will be useful to observe the form of
the image of $R_{\mathcal{C},{S}}$ by linear changes of variable.
It is worth noting that $\mathbf{R}_{F,\mathbf{S}}$ can be rewritten
$$\mathbf{R}_{F,\mathbf{S}}=id \wedge\left[\Delta_\mathbf{I}F\Delta_\mathbf{J}F\cdot\mathbf{S}-
\Delta_\mathbf{S}F\Delta_\mathbf{I}F\cdot\mathbf{J}-\Delta_\mathbf{S}F\Delta_\mathbf{J}F\cdot\mathbf{I}\right]. $$

\begin{prop}
Let $M\in GL(\mathbf{V})$. We have
$$\mathbf{R}_{F,\mathbf{S}}\circ M=Com(M)\cdot \mathbf{R}_{F\circ M,M^{-1}(\mathbf{S})}^{(M^{-1}(\mathbf{I}),M^{-1}(\mathbf{J}))}, $$
with $Com(M):=det(M)\cdot{}^tM^{-1}$ and
$$\mathbf{R}_{G,\mathbf{S'}}^{(\mathbf{A},\mathbf{B})}:=id 
\wedge\left[\Delta_\mathbf{A}G\Delta_\mathbf{B}G\cdot\mathbf{S'}-\Delta_{\mathbf{S'}}G\Delta_\mathbf{A}G\cdot\mathbf{B}-\Delta_{\mathbf{S'}}G\Delta_\mathbf{B}G\cdot\mathbf{A}\right]. $$
\end{prop}
\begin{proof}
We use $M(\mathbf{u})\wedge M(\mathbf{v})=(Com(M))(\mathbf{u}\wedge \mathbf{v})$ and
$\Delta_{M(\mathbf{u})}(F)(M(\mathbf{P}))=\Delta_\mathbf{u}(F\circ M)(\mathbf{P})$.
\end{proof}
We write $\Pi:\mathbf{V}\setminus\{\mathbf{0}\}\rightarrow\mathbb P^2$
for the canonical projection, $P_0[0:0:1]\in\mathbb P^2$ and $\mathbf{P}_0(0,0,1)\in\mathbf{V}$.
Let $m_1$ be a base point of $\mathcal C$ 
and $M\in GL(\mathbf{V})$
be such that $\Pi(M(\mathbf{P}_0))=m_1$ and such that
the tangent cone of $V(F\circ M)$ at $P_0$ does not contain $V(x)$.
Let $\mu_{m_1}$ be the multiplicity of $m_1$ in $\mathcal C$ ($m_1$ is a singular
point of $\mathcal C$ if and only if $\mu_{m_1}>1$). 
Then, for every $a\in\mathbb P^2$, writing $a':=  M^{-1}(a)$, we have
\begin{eqnarray*}
i_{m_1}(\mathcal C, {\mathcal P}^{(r)}_{S,a}(\mathcal C))
   &=& i_{m_1}(\mathcal C,V(\langle \mathbf{a}, \mathbf{R}_{F,\mathbf{S}}(\cdot)\rangle))\\
   &=& i_{P_0}(V(F\circ M),V(\langle \mathbf{a},
     \mathbf{R}_{F,\mathbf{S}}\circ M(\cdot)\rangle))\\
   &=& i_{P_0}(V(F\circ M),V(\langle \mathbf{a'},
     \mathbf{R}_{F\circ M,M^{-1}(\mathbf{S})}^{(M^{-1}(\mathbf{I}),M^{-1}(\mathbf{J}))}(\cdot)\rangle))\\
   &=& \sum_{\mathcal B\in Branch_{P_0}(V(F\circ M))}
   i_{P_0}(\mathcal B,V(\langle \mathbf{a'},
   \mathbf{R}_{F\circ M,M^{-1}(\mathbf{S})}^{(M^{-1}(\mathbf{I}),M^{-1}(\mathbf{J}))}(\cdot)\rangle)),
\end{eqnarray*}
where $Branch_{P_0}(V(F\circ M))$ is the set of branches of $V(F\circ M)$ at 
$P_0$. 
The last equality comes from Proposition \ref{formulemultiplicite} proved in
appendix (see formula (\ref{branche})).
Let $b$ be the number of such branches. Of course, $b=1$ for non-singular
points. Writing $e_{\mathcal B}$ for the multiplicity of the branch $\mathcal B$,
we have $\mu_{m_1}=\sum_{\mathcal B\in Branch_{P_0}(V(F\circ M))}e_{\mathcal B}$.
Let us write ${\mathbb C}\langle x^{\frac 1N}\rangle$ and ${\mathbb C}\langle x^{\frac 1N},y
  \rangle$ for the rings of convergent power series of $x^{\frac 1N},y$.
Let $\mathbb C\langle x^*\rangle:=\bigcup_{N\ge 1}{\mathbb C}\langle x^{\frac 1N}\rangle$
and $\mathbb C\langle x^*,y\rangle:=\bigcup_{N\ge 1}{\mathbb C}\langle x^{\frac 1N},y\rangle$.
For every $h=\sum_{q\in\mathbb Q_+}a_q x^q\in{\mathbb C}\langle x^*\rangle$, 
we define the valuation of $h$ as follows:
$$val(h):=val_x(h(x)):=\min\{q\in\mathbb Q_+,\ a_q\ne 0\}.$$
Let $\mathcal B$ be a branch of $V(F\circ M)$ at $P_0$.
We precise that $\mathcal B_0=M(\mathcal B)\subset\mathbb P^2$ is a branch of
$\mathcal C$ at $m_1$.
Let $\mathbf{A}(x_A,y_A,z_A):=M^{-1}(\mathbf{I})$, $\mathbf{B}(x_B,y_B,z_B):=M^{-1}(\mathbf{J})$ and $\mathbf{S'}:=M^{-1}(\mathbf{S})$.
Let $\mathcal T_{\mathcal B}$ be the tangent line to $\mathcal B$ at $P_0$. The branch $\mathcal B$ can be splitted in $e_\mathcal B$ 
pro-branches with equations $y=g_{i,\mathcal B}(x)$ in the chart $z=1$ 
(for $i\in\{1,...,e_{\mathcal B}\}$) 
with $g_i\in \mathbb C\langle x^*\rangle$ having (rational) valuation larger than or equal to 1 
(so $g_i'(0)=0$). For $j\in\{1,...,e_{\mathcal B'}\}$,
consider also the equations $y=g_{j,\mathcal B'}(x)$ (in the chart $z=1$)
of the pro-branches $\mathcal V_{j,\mathcal B'}$ 
for each branch $\mathcal B'\in Branch_{P_0}(V(F\circ M))$.
This notion of pro-branches comes from the combination of the
Weierstrass and of the Puiseux theorems. It has been used namely by Halphen in \cite{Halphen} and
by Wall in \cite{Wall}. One can also see \cite{fredsoaz1}. 
There exists a unit $U$ of $\mathbb C\langle x,y\rangle$ such that the following equality holds 
true in $\mathbb C\langle x^*,y\rangle$
$$F(M(x,y,1))=U(x,y)\prod_{\mathcal B'\in Branch_{P_0}(V(F\circ M))}
   \prod_{j=1}^{e_{\mathcal B'}} (y-g_{j,\mathcal B'}(x)).$$
For a generic $a$ (with $a':=M^{-1}(a)$), using 
(\ref{branche2})), we obtain
\begin{eqnarray*}
i_{P_0}(\mathcal B,V(\langle \mathbf{a'},
   \mathbf{R}_{F\circ M,\mathbf{S'}}^{(\mathbf{A},\mathbf{B})}(\cdot)\rangle))
  &=&\sum_i val_x\left( \langle \mathbf{a'},
     \mathbf{R}_{F\circ M,\mathbf{S'}}^{(\mathbf{A},\mathbf{B})}(x,g_{i,\mathcal B}(x),1)\rangle\right)\\
  &=&\sum_i  \min_{j=1,2,3}val_x\left(\left[\mathbf{R}_{F\circ M,\mathbf{S'}}^{(\mathbf{A},\mathbf{B})}
     (x,g_{i,\mathcal B}(x),1)\right]_j\right).
\end{eqnarray*}
Hence Formula (\ref{formule1}) becomes
\begin{equation}\label{EQ**}
\mbox{mclass}(\Sigma_{S}(\mathcal C))
  := d(2d-1)-\sum_{m_1\in\mathcal C}\sum_{\mathcal B}
  \sum_{i=1}^{e_{\mathcal B}}
  \min_{j=1,2,3}val_x\left(\left[\bf{R}_{F\circ M,\mathbf{S'}}^{(\mathbf{A},\mathbf{B})}
     (x,g_{i,\mathcal B}(x),1)\right]_j\right),
\end{equation}
where, for every $m_1\in\mathcal C$, $M$ depends on $m_1$ and is as above,
where the sum is over $\mathcal B\in Branch_{P_0}(V(F\circ M))$.
Due to Lemma 33 of \cite{fredsoaz1}, for every $\mathbf{P}
(x_P,y_P,z_P)\in\mathbf{V}\setminus\{\mathbf{0}\}$, we have
$$(\Delta_{M(\mathbf{P})}F)\circ M(x,g_{i,\mathcal B}(x),1)=\Delta_{\mathbf{P}}(F\circ M)(x,g_{i,\mathcal B}(x),1)
   =  D_{i,\mathcal B}(x)W_{\mathbf{P},i,\mathcal B}(x),$$
with
$$W_{\mathbf{P},i,\mathcal B}(x):=y_P -g_{i,\mathcal B}'(x)x_P+z_P(xg'_{i,\mathcal B}(x)-g_
          {i,\mathcal B}(x))$$
and with $D_{i,\mathcal B}(x):=U(x,g_{i,\mathcal B}(x))
   \prod_{\mathcal B'\in Branch_{P_0}(V(F\circ M))}
     \prod_{j=1,...,e_{\mathcal B'}:(\mathcal B',j)\ne(\mathcal B,i)}
        (g_{i,\mathcal B}(x)-g_{j,\mathcal B'}(x))$.
Hence we have
\begin{equation}\label{eqR}
\mathbf{R}_{F\circ M,\mathbf{S'}}^{(\mathbf{A},\mathbf{B})}(x,g_{i,\mathcal B}(x),1):=
(D_{i,\mathcal B})^2\cdot\hat R_{i,\mathcal B}(x)
\end{equation}
with
$$\hat R_{i,\mathcal B}(x):=\left(\begin{array}{c}x\\g_{i,\mathcal B}(x)\\
   1\end{array}\right)\wedge
   \left[W_{\mathbf{A},i,\mathcal B}(x)W_{\mathbf{B},i,\mathcal B}(x)\cdot\mathbf{S'}-W_{
\mathbf{S'},i,\mathcal B}(x)W_{\mathbf{A},i,
     \mathcal B}(x)\cdot\mathbf{B}-W_{\mathbf{S'},i,\mathcal B}(x)W_{\mathbf{B},i,\mathcal B}(x)\cdot\mathbf{A}\right]. $$
First, with the notations of \cite{fredsoaz1} (since $U(0,0)\ne 0$), we have
$$\sum_{\mathcal B\in Branch_{P_0}(V(F\circ M))}
     \sum_{i=1}^{e_{\mathcal B}}val (D_{i,\mathcal B})=V_{m_1},$$
(which is null if $m_1$ is a nonsingular point of $\mathcal C$).
Second, writing
$h_{m_1,i,\mathcal B} := \min(val([\hat R_{i,\mathcal B}]_j),\ j=1,2,3)$,
we observe that, due to Proposition 29 and to Remark 34 of \cite{fredsoaz1}, the quantity $\sum_{i=1}^{e_\mathcal B} h_{m_1,i,\mathcal B}$ only depends on $m_1$ and on the branch $\mathcal B_0=M(\mathcal B)$ of $\C$ at $m_1$ (it does not depend on the choice of $M\in GL(\mathbf{V})$ such that $\Pi(M(P_0))=m_1$ and such that $V(x)$ is not tangent to
$M^{-1}(\mathcal B_0)$). Hence we write
$$h_{m_1,\mathcal B_0}:=\sum_{i=1}^{e_{\mathcal{B}}}h_{m_1,i,\mathcal{B}}.$$
With these notations, due to (\ref{eqR}), formula (\ref{EQ**}) becomes
$${mclass(\Sigma_S(\mathcal C))=2d(d-1)+d-2\sum_{m_1\in Sing(\mathcal C)}V_{m_1}
    -\sum_{m_1\in\mathcal C}\sum_{\mathcal{B}_0\in Branch_{m_1}(\C)}h_{m_1,\mathcal B_0}}.$$
Moreover, as noticed in \cite{fredsoaz1}, we have 
$d(d-1)-\sum_{m_1\in Sing(\mathcal C)}V_{m_1}=d^\vee,$
where $d^\vee$ is the class of $\mathcal C$.
Therefore, we get
\begin{equation}\label{EQfin}
mclass(\Sigma_S(\mathcal C))=2d^\vee+d-\sum_{m_1\in\mathcal C}\sum_{\mathcal{B}_0\in Branch_{m_1}(\C)}h_{m_1,\mathcal B_0}.
\end{equation}
Theorem \ref{THM1} will come directly from the computation of $h_{m_1,i,\mathcal B}$
given in following result.
\begin{lem}\label{LEM}
Let $m_1\in\mathcal C$ and $\mathcal B_0\in Branch_{m_1}(\C)$.
Writing $\mathcal T_{m_1}\mathcal{B}_0$ for the tangent line
to $\mathcal B_0$ at $m_1$, $i_{m_1}(\mathcal{B}_0,\mathcal T_{m_1}\mathcal{B}_0)$ for the intersection number of $\mathcal B_0$
with $\mathcal T_{m_1}\mathcal{B}_0$ at $m_1$ and $e_{\mathcal B_0}$
for the multiplicity of $\mathcal B_0$, we have
\begin{enumerate}
\item $h_{m_1,\mathcal B_0}=0$ 
if $I, J,S\not\in\mathcal T_{m_1}\mathcal{B}_0$.
\item  
$h_{m_1,\mathcal B_0}=0$ 
if $\#(\mathcal T_{m_1}\mathcal{B}_0\cap\{ I, J, S\})=1$ and
$m_1\not\in\{  I, J, S\}$.
\item  $h_{m_1,\mathcal B_0}=e_{\mathcal{B}_0}$ 
if $\#(\mathcal T_{m_1}\mathcal{B}_0\cap\{ I, J, S\})=1$ and
$m_1\in\{  I, J, S\}$.
\item
$h_{m_1,\mathcal B_0}=i_{m_1}(\mathcal{B}_0,\mathcal T_{m_1}\mathcal{B}_0)
    +\min(i_{m_1}(\mathcal{B}_0,\mathcal T_{m_1}\mathcal{B}_0)-2e_{\mathcal{B}_0},0)$ 
if $\mathcal T_{m_1}\mathcal{B}_0=( I S)$, $ J\not\in
   \mathcal T_{m_1}\mathcal{B}_0$ and $m_1\not\in\{ I, S\}$.

$h_{m_1,\mathcal B_0}=i_{m_1}(\mathcal{B}_0,\mathcal T_{m_1}\mathcal{B}_0)
    +\min(i_{m_1}(\mathcal{B}_0,\mathcal T_{m_1}\mathcal{B}_0)-2e_{\mathcal{B}_0},0)$ 
if $\mathcal T_{m_1}\mathcal{B}_0=( J S)$, $ I\not\in
   \mathcal T_{m_1}\mathcal{B}_0$ and $m_1\not\in\{ J, S\}$.
\item
$h_{m_1,\mathcal B_0}=i_{m_1}(\mathcal{B}_0,\mathcal T_{m_1}\mathcal{B}_0)$ 
if $\mathcal T_{m_1}\mathcal{B}_0=( I S)$, $ J\not\in
   \mathcal T_{m_1}\mathcal{B}_0$ and $m_1\in\{ I, S\}$.

$h_{m_1,\mathcal B_0}=i_{m_1}(\mathcal{B}_0,\mathcal T_{m_1}\mathcal{B}_0)$ 
if $\mathcal T_{m_1}\mathcal{B}_0=( J S)$, $ I\not\in
   \mathcal T_{m_1}\mathcal{B}_0$ and $m_1\in\{ J, S\}$.
\item
$h_{m_1,\mathcal B_0}=i_{m_1}(\mathcal{B}_0,\mathcal T_{m_1}\mathcal{B}_0)
  -e_{\mathcal{B}_0}$ 
 if $\mathcal T_{m_1}\mathcal{B}_0=( I J)$, that $ S\not\in
   \mathcal T_{m_1}\mathcal{B}_0$ and
$m_1\not\in\{ I, J\}$.
\item $h_{m_1,\mathcal B_0}=
  i_{m_1}(\mathcal{B}_0,\mathcal T_{m_1}\mathcal{B}_0)$ 
if $\mathcal T_{m_1}\mathcal{B}_0=(I J)$, that $ S\not\in
   \mathcal T_{m_1}\mathcal{B}_0$ and
$m_1\in\{ I, J\}$.
\item $h_{m_1,\mathcal B_0}=
  2i_{m_1}(\mathcal{B}_0,\mathcal T_{m_1}\mathcal{B}_0)-2e_{\mathcal{B}_0}$
if $ I, J, S\in\mathcal T_{m_1}\mathcal{B}_0$ and $m_1\not\in\{
    I, J, S\}$.
\item $h_{m_1,\mathcal B_0}=
  2i_{m_1}(\mathcal{B}_0,\mathcal T_{m_1}\mathcal{B}_0)-e_{\mathcal{B}_0}$
if $ I, J, S\in\mathcal T_{m_1}\mathcal{B}_0$ and 
$m_1\in\{ I, J\}$.
\item $h_{m_1,\mathcal B_0}=
  2i_{m_1}(\mathcal{B}_0,\mathcal T_{m_1}\mathcal{B}_0)-e_{\mathcal{B}_0}$
if $ I,J, S\in\mathcal T_{m_1}\mathcal{B}_0$, 
$m_1= S$ and $i_{m_1}(\mathcal{B}_0,\mathcal T_{m_1}\mathcal{B}_0)\ne 2e_{\mathcal{B}_0}$.
\item $h_{m_1,\mathcal B_0}=
  e_{\mathcal{B}_0}(1+\min(\beta_1,3))$
if $ I, J, S\in\mathcal T_{m_1}\mathcal{B}_0$, $m_1= S$
and $i_{m_1}(\mathcal{B}_0,\mathcal T_{m_1}\mathcal{B}_0)= 2e_{\mathcal{B}_0}$, 
$e_{\mathcal{B}_0}\beta_1=i_{m_1}({\mathcal{B}_0},Osc_{m_1}(\mathcal{B}_0))$, where
$Osc_{m_1}(\mathcal{B}_0)$ is any osculating smooth algebraic curve to $\mathcal{B}_0$ at $m_1$ 
(the last formula of $h_{m_1,\mathcal B_0}$ holds true if we replace $e_{\mathcal B_0}\beta_1$ by the first characteristic exponent of $\mathcal B_0$ non multiple of $e_{\mathcal B_0}$, see \cite{Zariski}).
\end{enumerate}
\end{lem}
\begin{proof}
We take $M$ such that $\mathcal T_{\mathcal B}=V(y)$ (with
$\mathcal B=M^{-1}(\mathcal B_0)$).
To simplify notations, we ommit indices $\mathcal B$ in 
$W_{\mathbf{P},i,\mathcal{B}}$ and consider $i\in\{1,...,e_{\mathcal B}\}$.
\begin{itemize}
\item Suppose that $ I, J,S \not \in\mathcal T_{m_1}\mathcal{B}_0$.
Then $W_{\mathbf{B},i}(0)=y_B\ne 0$, $W_{\mathbf{A},i}(0)=y_A\ne 0$ and $W_{\mathbf{S'},i}(0)=y_{S'}\ne 0$
so
\begin{eqnarray*}
\hat R_i(0)&=&\left(\begin{array}{c}0\\0\\1\end{array}\right)
      \wedge [y_Ay_B\cdot\mathbf{S'}-y_Ay_{S'}\cdot\mathbf{B}-y_By_{S'}\cdot\mathbf{A}]\\
&=&\left(\begin{array}{c}y_Ay_By_{S'}\\y_Ay_Bx_{S'}-y_Ay_{S'}x_B-y_By_{S'}x_A\\0\end{array}\right).
\end{eqnarray*}
Hence $h_{m_1,i,\mathcal B}=0$ and the sum over $i=1,...,e_{\mathcal B}$ of these
quantities is equal to $0$.

\item  Suppose $ I\in\mathcal T_{m_1}\mathcal{B}_0$, $ J,\mathcal S
   \not \in\mathcal T_{m_1}\mathcal{B}_0$ and
$m_1\ne I$.
Take $M$ such that $\mathbf{S'}(0,1,0)$, $\mathbf{A}(1,0,0)$, $y_B\ne 0$.
We have
$W_{\mathbf{B},i}(0)=y_B$, $W_{\mathbf{A},i}(0)=0$ and $W_{\mathbf{S'},i}(0)=1$
and so
$
\hat R_i(0)=\left(\begin{array}{c}0\\0\\1\end{array}\right)
      \wedge \left(\begin{array}{c}-y_B\\0\\0\end{array}\right)
   =\left(\begin{array}{c}0\\-y_B\\0\end{array}\right).$
Hence $h_{m_1,i,\mathcal B}=0$ and the sum over $i=1,...,e_{\mathcal B}$ of these
quantities is equal to $0$.

\item  Suppose $ I\in\mathcal T_{m_1}\mathcal{B}_0$, $J,S
   \not \in\mathcal T_{m_1}\mathcal{B}_0$ and
$m_1=  I$.
Take $M$ such that $\mathbf{S'}(0,1,0)$, $\mathbf{A}(0,0,1)$, $y_B\ne 0$.
We have
$W_{\mathbf{B},i}(x)=y_B-g_i'(x)x_B+z_B(xg_i'(x)-g_i(x))$, $W_{\mathbf{A},i}(x)=xg_i'(x)-g_i(x)$ and $W_{\mathbf{S'},i}(x)=1$
and so
\begin{eqnarray*}
\hat R_i(x)&=&\left(\begin{array}{c}x\\g_i(x)\\1\end{array}\right)
      \wedge \left(\begin{array}{c}-(xg_i'(x)-g_i)x_B\\
    (xg_i'(x)-g_i)(-g_i'(x)x_B+z_B(xg_i'(x)-g_i(x)))\\
     -y_B+g_i'(x)x_B-2z_B(xg_i'(x)-g_i(x))\end{array}\right)\\
   &=&\left(\begin{array}{c}-y_Bg_i(x)+x(g'_i(x))^2x_B-z_B((xg'_i(x))^2-(g_i(x))^2)\\
        -x_B(2xg'_i(x)-g_i(x))+xy_B+2xz_B(xg'_i(x)-g_i(x))\\
        -x_B(xg_i'(x)-g_i(x))^2+z_B(xg'_i(x)-g_i(x))
  \end{array}\right),
\end{eqnarray*}
the valuation of the coordinates of which are larger than or equal to $1$
and the valuation of the second coordinate is $1$.
Hence $h_{m_1,i,\mathcal B}=1$ and the sum over $i=1,...,e_{\mathcal B}=e_{\mathcal B_0}$ of these
quantities is equal to $e_{\mathcal B_0}$.

\item  Suppose $S\in\mathcal T_{m_1}\mathcal{B}_0$, $ I, J
   \not \in\mathcal T_{m_1}\mathcal{B}_0$ and
$m_1\ne S$.
Take $M$ such that $\mathbf{A}(0,1,0)$, $\mathbf{S'}(1,0,0)$, $y_B\ne 0$.
We have
$W_{\mathbf{B},i}(0)=y_B\ne 0$, $W_{\mathbf{S'},i}(0)=0$ and $W_{\mathbf{A},i}(0)=1$
and so
$\hat R_i(0)=\left(\begin{array}{c}0\\0\\1\end{array}\right)
      \wedge \left(\begin{array}{c}y_B\\0\\0\end{array}\right)
=\left(\begin{array}{c}0\\y_B\\0\end{array}\right).$
Hence $h_{m_1,i,\mathcal B}=0$ and the sum over $i=1,...,e_{\mathcal B}$ of these
quantities is equal to $0$.

\item  Suppose $m_1= S$ and $ I, J
   \not \in\mathcal T_{m_1}\mathcal{B}_0$.
Take $M$ such that $\mathbf{S'}(0,0,1)$, $\mathbf{A}(0,1,0)$, $y_B\ne 0$.
We have
$W_{\mathbf{B},i}(x)=y_B-g_i'(x)x_B+z_B(xg_i'(x)-g_i(x))$, $W_{\mathbf{S'},i}(x)=xg_i'(x)-g_i(x)$ and 
$W_{\mathbf{A},i}(x)=1$
and so
\begin{eqnarray*}
\hat R_i(x)&=&\left(\begin{array}{c}x\\g_i(x)\\1\end{array}\right)
      \wedge \left(\begin{array}{c}-(xg_i'(x)-g_i)x_B\\
    -(xg_i'(x)-g_i(x))(2y_B-g_i'(x)x_B+z_B(xg_i'(x)-g_i(x)))\\
     y_B-g_i'(x)x_B\end{array}\right)\\
   &=&\left(\begin{array}{c}g_i(x)(y_B-g_i'(x)x_B)+(xg_i'(x)-g_i(x))(2y_B-g_i'(x)x_B+z_B
      (xg_i'(x)-g_i(x)))\\
       -(xg_i'(x)-g_i(x))x_B-x(y_B-g_i'(x)x_B)\\
        -(xg_i'(x)-g_i(x))(2xy_B-g_i'(x)xx_B+z_Bx(xg_i'(x)-g_i(x))+g_i(x)(xg_i'(x)-g_i(x))
         x_B)  \end{array}\right),
\end{eqnarray*}
the valuation of the coordinates of which are larger than or equal to $1$
and the valuation of the second coordinate is $1$.
Hence $h_{m_1,i,\mathcal B}=1$ and the sum over $i=1,...,e_{\mathcal B}$ of these
quantities is equal to $e_{\mathcal B_0}$.

\item Suppose $\mathcal T_{m_1}\mathcal{B}_0=( IS)$, $ J\not\in
   \mathcal T_{m_1}\mathcal{B}_0$ and
$m_1\not\in\{ I, S\}$.
Take $M$ such that $\mathbf{S'}(1,0,0)$, $\mathbf{B}(0,1,0)$, $y_A=0$, $x_A\ne 0$, $z_A\ne 0$. We have
$W_{\mathbf{S'},i}(x)=-g_i'(x)$, $W_{\mathbf{A},i}(x)=-g_i'(x)x_A+z_A(xg_i'(x)-g_i(x))$ and $W_{\mathbf{B},i}(x)=1$
and so
\begin{eqnarray*}
\hat R_i(x)&=&\left(\begin{array}{c}x\\g_i(x)\\1\end{array}\right)
      \wedge \left(\begin{array}{c}z_A(xg_i'(x)-g_i(x))\\
    -(g_i'(x))^2x_A+g_i'(x)(xg_i'(x)-g_i(x))z_A\\
       g_i'(x)z_A\end{array}\right)  \\
    &=&\left(\begin{array}{c}
    g_i(x)g_i'(x)z_A+(g_i'(x))^2x_A-g_i'(x)(xg_i'(x)-g_i(x))z_A\\
    -g_i(x)z_A\\  -x(g_i'(x))^2x_A+(xg_i'(x)-g_i(x))^2z_A\end{array}\right),
\end{eqnarray*}
the valuation of the coordinates of which are respectively $2 val(g_i)-2$, $val(g_i)$ 
and $2 val(g_i)-1$.
Hence $h_{m_1,i,\mathcal B}=val(g_i)+\min(val(g_i)-2,0)$ 
and the sum over $i=1,...,e_{\mathcal B}$ of these
quantities is equal to $i_{m_1}(\mathcal B_0,\mathcal T_{m_1}{\mathcal B_0})+
\min(i_{m_1}(\mathcal B_0,\mathcal T_{m_1}\mathcal B_0)-2e_{\mathcal B_0},0)   $.
\item Suppose $\mathcal T_{m_1}\mathcal{B}_0=( I S)$, $ J\not\in
   \mathcal T_{m_1}\mathcal{B}_0$ and
$m_1= I$.
Take $M$ such that $\mathbf{S'}(1,0,0)$, $\mathbf{B}(0,1,0)$, $\mathbf{A}(0,0,1)$. We have
$W_{\mathbf{S'},i}(x)=-g_i'(x)$, $W_{\mathbf{A},i}(x)=xg_i'(x)-g_i(x)$ and $W_{\mathbf{B},i}(x)=1$
and so
$$
\hat R_i(x)=\left(\begin{array}{c}x\\g_i(x)\\1\end{array}\right)
      \wedge \left(\begin{array}{c}xg_i'(x)-g_i(x)\\g_i'(x)(xg_i'(x)-g_i(x))\\
       g_i'(x)\end{array}\right)
    =\left(\begin{array}{c}
    g'_i(x)(2g_i(x)-xg_i'(x))\\
    -g_i(x)\\  (xg_i'(x)-g_i(x))^2\end{array}\right),
$$
the valuation of the coordinates of which are larger than or equal to $val(g_i)$,
the second coordinate has valuation $val(g_i)$.
Hence $h_{m_1,i,\mathcal B}=val(g_i)$ and the sum over $i=1,...,e_{\mathcal B}$ of these
quantities is equal to $i_{m_1}(\mathcal B_0,\mathcal T_{m_1}{\mathcal B_0})$.

\item Suppose $\mathcal T_{m_1}\mathcal{B}_0=( IS)$, $ J\not\in
   \mathcal T_{m_1}\mathcal{B}_0$ and
$m_1=S$.
Take $M$ such that $\mathbf{A}(1,0,0)$, $\mathbf{B}(0,1,0)$, $\mathbf{S'}(0,0,1)$. We have
$W_{\mathbf{S'},i}(x)=xg_i'(x)-g_i(x)$, $W_{\mathbf{A},i}(x)=-g_i'(x)$ and $W_{\mathbf{B},i}(x)=1$
and so
$$
\hat R_i(x)=\left(\begin{array}{c}x\\g_i(x)\\1\end{array}\right)
      \wedge \left(\begin{array}{c}-(xg_i'(x)-g_i(x))\\g_i'(xg_i'(x)-g_i(x))\\
       -g_i'(x)\end{array}\right)
  =\left(\begin{array}{c}
    -x(g_i'(x))^2\\ g_i(x)  \\  (xg_i'(x))^2-(g_i(x))^2\end{array}\right),
$$
the valuation of the coordinates of which being larger than or equal to $val(g_i)$
and the valuation of the second coordinate is equal to $val(g_i)$.
Hence $h_{m_1,i,\mathcal B}=val(g_i)$ and the sum over $i=1,...,e_{\mathcal B}$ of these
quantities is equal to $i_{m_1}(\mathcal B_0,\mathcal T_{m_1}{\mathcal B_0})$.

\item  Suppose $\mathcal T_{m_1}\mathcal{B}_0=( I J)$, $S\not\in
   \mathcal T_{m_1}\mathcal{B}_0$ and
$m_1\not\in\{ I, J\}$.
Take $M$ such that $\mathbf{S'}(0,1,0)$, $\mathbf{B}(1,0,0)$, $y_A=0$, $x_A\ne 0$, $z_A\ne 0$. We have
$W_{\mathbf{B},i}(x)=-g_i'(x)$, $W_{\mathbf{A},i}(x)=-g_i'(x)x_A+z_A(xg_i'(x)-g_i(x))$ and $W_{\mathbf{S'},i}(x)=1$
and so
\begin{eqnarray*}
\hat R_i(x)&=&\left(\begin{array}{c}x\\g_i(x)\\1\end{array}\right)
      \wedge \left(\begin{array}{c}2g_i'(x)x_A-z_A(xg_i'(x)-g_i(x))\\
   (g_i'(x))^2x_A-g_i'(x)(xg_i'(x)-g_i(x))z_A\\
       g_i'(x)z_A\end{array}\right)\\
   &=&\left(\begin{array}{c}
    (g_i'(x))^2(xz_A-x_A)\\
    2g_i'(x)x_A-z_A(2xg_i'(x)-g_i(x))\\  -z_A(xg_i'(x)-g_i(x))^2
        +x_Ag_i'(x)(xg_i'(x)-2g_i(x))\end{array}\right),
\end{eqnarray*}
the valuation of the coordinates of which are respectively $2 val(g_i)-2$, $val(g_i)-1$ 
and larger than $val(g_i)$.
Hence $h_{m_1,i,\mathcal B}=val(g_i)-1$ and the sum over $i=1,...,e_{\mathcal B}$ of these
quantities is equal to $i_{m_1}(\mathcal B_0,\mathcal T_{m_1}\mathcal{B}_0)-e_{\mathcal B_0}$.

\item  Suppose that $\mathcal T_{m_1}\mathcal{B}_0=( I J)$, that $S\not\in
   \mathcal T_{m_1}\mathcal{B}_0$ and $m_1= I$.
Take $M$ such that $\mathbf{S'}(0,1,0)$, $\mathbf{B}(1,0,0)$, $\mathbf{A}(0,0,1)$. We have
$W_{\mathbf{B},i}(x)=-g_i'(x)$, $W_{\mathbf{A},i}(x)=xg_i'(x)-g_i(x)$ and $W_{\mathbf{S'},i}(x)=1$
and so
$$
\hat R_i(x)=\left(\begin{array}{c}x\\g_i(x)\\1\end{array}\right)
      \wedge \left(\begin{array}{c}-(xg_i'(x)-g_i(x))\\
   -g_i'(x)(xg_i'(x)-g_i(x))\\
       g_i'(x)\end{array}\right)
   =\left(\begin{array}{c}
    x(g_i'(x))^2\\
    -(2xg_i'(x)-g_i(x))\\  -(xg_i'(x)-g_i(x))^2\end{array}\right),
$$
the valuation of the coordinates of which being larger than or equal to $val(g_i)$ and
the valuation of the second coordinate is equal to $val(g_i)$.
Hence $h_{m_1,i,\mathcal B}=val(g_i)$ and the sum over $i=1,...,e_{\mathcal B}$ of these
quantities is equal to $i_{m_1}(\mathcal B_0,\mathcal T_{m_1}{\mathcal B_0})$.

\item  Suppose that $ I, J,S\in\mathcal T_{m_1}\mathcal{B}_0$ and
$m_1\not\in\{ I, J, S\}$.
Take $M$ such that $\mathbf{S'}(1,0,0)$, $y_A=y_B=0$, $x_A\ne 0$, $z_A\ne 0$, $x_B\ne 0$,
$z_B\ne 0$, $x_Az_B\ne x_Bz_A$. 
We have
$W_{\mathbf{S'},i}(x)=-g_i'(x)$, $W_{\mathbf{A},i}(x)=-g_i'(x)x_A+z_A(xg_i'(x)-g_i(x))$ and $W_{\mathbf{B},i}(x)=
   -g_i'(x)x_B+z_B(xg_i'(x)-g_i(x))$
and so
\begin{eqnarray*}
\hat R_i(x)&=&\left(\begin{array}{c}x\\g_i(x)\\1\end{array}\right)
      \wedge \left(\begin{array}{c}-x_A(g_i'(x))^2x_B+z_Az_B(xg_i'(x)-g_i(x))^2\\0
     \\-(g_i'(x))^2(xz_B+x_Bz_A)_A+2z_Az_Bg_i'(x)(xg_i'(x)-g_i(x))\end{array}\right)\\
    &=&\left(\begin{array}{c}
      -g_i(x)(g_i'(x))^2(xz_B+x_Bz_A)_A+2z_Az_Bg_i(x)g_i'(x)(xg_i'(x)-g_i(x))\\
      -x_A(g_i'(x))^2x_B+z_Az_B(xg_i'(x)-g_i(x))^2-x[....]\\   
      x_Ag_i(x)(g_i'(x))^2x_B-z_Az_Bg_i(x)(xg_i'(x)-g_i(x))^2
     \end{array}\right),
\end{eqnarray*}
the valuation of the coordinates of which are larger than or equal to $2 val(g_i)-2$, 
the valuation of the second coodinate is $2 val(g_i)-2$.
Hence $h_{m_1,i,\mathcal B}=2val(g_i)-2$ and the sum over $i=1,...,e_{\mathcal B}$ of these
quantities is equal to $2i_{m_1}(\mathcal B_0,\mathcal T_{m_1}{\mathcal B_0})-2e_{\mathcal B_0}$.

\item Suppose that $ I, J,S\in\mathcal T_{m_1}\mathcal{B}_0$ and
$m_1= J$.
Take $M$ such that $\mathbf{B}(0,0,1)$, $\mathbf{S'}(1,0,0)$, $y_A=0$, $x_A\ne 0$ and $z_A\ne 0$.
We have
$W_{\mathbf{S'},i}(x)=-g_i'(x)$, $W_{\mathbf{A},i}(x)=-g_i'(x)x_A+z_A(xg_i'(x)-g_i(x))$ and $W_{\mathbf{B},i}(x)=
   xg_i'(x)-g_i(x)$
and so
\begin{eqnarray*}
\hat R_i(x)&=&\left(\begin{array}{c}x\\g_i(x)\\1\end{array}\right)
      \wedge \left(\begin{array}{c}z_A(xg_i'(x)-g_i(x))^2\\0
      \\-x_A(g_i'(x))^2+2z_A(xg_i'(x)-g_i(x))g_i'(x)\end{array}\right)\\
  &=&\left(\begin{array}{c}
    g_i(x)g_i'(x)(-g_i'(x)x_A+2z_A(xg_i'(x)-g_i(x)))\\
    z_A(xg_i'(x)-g_i(x))^2-xg_i'(x)(-g_i'(x)x_A+2z_A(xg_i'(x)-g_i(x)))\\
    -g_i(x)z_A(xg_i'(x)-g_i(x))^2\end{array}\right),
\end{eqnarray*}
the valuation of the coordinates of which are larger than or equal to $2 val(g_i)-1$
and the valuation of the second coordinate is $2 val(g_i)-1$.
Hence $h_{m_1,i,\mathcal B}=2val(g_i)-1$ and the sum over $i=1,...,e_{\mathcal B}$ of these
quantities is equal to $2i_{m_1}(\mathcal B_0,\mathcal T_{m_1}{\mathcal B_0})-e_{\mathcal B_0}$.

\item Suppose that $ I, J,S\in\mathcal T_{m_1}\mathcal{B}_0$ and
$m_1=S$.
Take $M$ such that $\mathbf{S'}(0,0,1)$, $\mathbf{B}(1,0,0)$, $y_A=0$, $x_A\ne 0$ and $z_A\ne 0$.
We have
$W_{\mathbf{B},i}(x)=-g_i'(x)$, $W_{\mathbf{A},i}(x)=-g_i'(x)x_A+z_A(xg_i'(x)-g_i(x))$ and $W_{\mathbf{S'},i}(x)=
   xg_i'(x)-g_i(x)$
and so
\begin{eqnarray*}
\hat R_i(x)&=&\left(\begin{array}{c}x\\g_i(x)\\1\end{array}\right)
      \wedge \left(\begin{array}{c}2x_Ag_i'(x)(xg_i'(x)-g_i(x))-z_A(xg_i'(x)-g_i(x))^2\\0
      \\ (g_i'(x))^2x_A\end{array}\right)\\
   &=&\left(\begin{array}{c}
    g_i(x)(g_i'(x))^2x_A\\
   x_Ag_i'(x)(xg_i'(x)-2g_i(x))-z_A(xg_i'(x)-g_i(x))^2\\
   -2x_Ag_i(x)g_i'(x)(xg_i'(x)-g_i(x))+z_Ag_i(x)(xg_i'(x)-g_i(x))^2\end{array}\right).
\end{eqnarray*}
The valuation of the first coordinate is $3 val(g_i)-2$ is smaller than or equal to
the valuation of the third coordinate.

If $val(g_i)\ne 2$, the valuation of the second coordinate is $2 val(g_i)-1$;
hence $h_{m_1,i,\mathcal B}=2val(g_i)-1$ and the sum over $i=1,...,e_{\mathcal B}$ of these
quantities is equal to $2i_{m_1}(\mathcal B_0,\mathcal T_{m_1}{\mathcal B_0})-e_{\mathcal B_0}$.

Suppose now that $val(g_i)= 2$, then $3 val(g_i)-2=4$ and 
there exist $\alpha,\alpha_1\in\mathbb C$
and $\beta_1>2$ such that $g_i(x)=\alpha x^2+\alpha_1 x^{\beta_1}+...$.
Then, the second coordinate has the following form 
$(x_A2\alpha(\beta_1-2)x^{\beta_1+1}+...)+x^4(...).$
Therefore $h_{m_1,i,\mathcal B}=\min(\beta_1+1,4)$ 
and the sum over $i=1,...,e_{\mathcal B}$ of these
quantities is equal to $e_{\mathcal B_0}(1+\min(\beta_1,3))$.
\end{itemize}
\end{proof}
\begin{proof}[Proof of Theorem \ref{THM1}]
Recall that \refeq{EQfin} says
$$
mclass(\Sigma_S(\mathcal C))=2d^\vee+d-\sum_{m_1\in\mathcal C}\sum_{\mathcal{B}_0\in Branch_{m_1}(\C)}h_{m_1,\mathcal B_0}$$
and that the values of $h_{m_1,\mathcal B_0}$ have been
given in Lemma \ref{LEM}.
\begin{itemize}
\item Assume first $S\not\in\ell_\infty$. Then we have to sum
the $h_{m_1,\mathcal B_0}$ coming from Items 3, 4, 5, 6 and 7
of Lemma \ref{LEM}.

The sum of the $h_{m_1,\mathcal B_0}$ coming from Items 3 and 5 applied with
$m_1=S$ gives directly $g'$.

The sum of the $h_{m_1,\mathcal B_0}$ coming from Items 3, 5 and 7 applied with
$m_1\in\{I,J\}$ gives $f+\Omega_I(\C,\ell_\infty)+\Omega_J(\C,\ell_\infty)$.

The sum of the $h_{m_1,\mathcal B_0}$ coming from Item 6 gives $g-\Omega_I(\C,\ell_\infty)-\Omega_J(\C,\ell_\infty)$.

The sum of the $h_{m_1,\mathcal B_0}$ coming from Item 4 gives $2f'-q'$
(notice that 
$h_{m_1,\mathcal B_0}=2(i_{m_1}(\mathcal{B}_0,\mathcal T_{m_1}\mathcal{B}_0)-e_{\mathcal{B}_0})-(i_{m_1}(\mathcal{B}_0,\mathcal T_{m_1}\mathcal{B}_0)-2e_{\mathcal{B}_0})\mathbf 1_{i_{m_1}(\mathcal{B}_0,\mathcal T_{m_1}\mathcal{B}_0)\ge 2 e_{\mathcal{B}_0}}$).
\item Assume first $S\not\in\ell_\infty$. Then we have to sum
the $h_{m_1,\mathcal B_0}$ coming from Items 3, 8, 9, 10 and 11
of Lemma \ref{LEM}.

The sum of the $h_{m_1,\mathcal B_0}$ coming from Items 3 (with $m_1=S$), 10 and 11 gives $2\Omega_S(\C,\ell_\infty)+\mu_S+c'(S)$.

The sum of the $h_{m_1,\mathcal B_0}$ coming from Items 3 and 9 applied with
$m_1\in\{I,J\}$ gives 
$2(\Omega_I(\C,\ell_\infty)+\Omega_J(\C,\ell_\infty))+\mu_I+\mu_J$.

The sum of the $h_{m_1,\mathcal B_0}$ coming from Item 8 gives
$2(g-\Omega_I(\C,\ell_\infty)-\Omega_J(\C,\ell_\infty)-\Omega_S(\C,\ell_\infty))$.
\end{itemize}
\end{proof}
%
%
%
%
%

%
%
%
%
%
%
\begin{appendix}

\section{Intersection numbers of curves and pro-branches}\label{append2}
The following result expresses the classical intersection number
$i_{m_1}(\mathcal C,\mathcal C')$ defined in \cite[p. 54]{Hartshorne}
thanks to the use of probranches.
\begin{prop}\label{formulemultiplicite}
Let $m\in\mathbb P^2$. 
Let $\mathcal C=V(F)$ and $\mathcal C'=V(F')$ be two algebraic plane curves
containing $m$, with homogeneous polynomials $F,F'\in 
\mathbb{C}[X,Y,Z]$. 
Let $M\in GL(\mathbb C^3)$ be such that $\Pi(M(\mathbf{P}_0))=m$ and such that
the tangent cones of $V(F\circ M)$ and of $V(F'\circ M)$ do not contain $X=0$.

Assume that $V(F\circ M)$ admits $b$ branches at $P_0$ and that its
$\beta$-th branch $\mathcal B_\beta$ has multiplicity $e_{\beta}$.
Assume that $V(F'\circ M)$ admits $b'$ branches at $P_0$ and that its
$\beta'$-th branch $\mathcal B'_{\beta'}$ has multiplicity $e'_{\beta'}$.

Then we have
$$i_{m}(\mathcal C,\mathcal C')
  =\sum_{\beta=1}^{b}\sum_{j=0}^{e_\beta-1}\sum_{\beta'=1}^{b'}\sum_{j'=0}^{e'_{\beta'}-1}
  val_{x}[h_{\beta}(\zeta^{j}x^{\frac{1}{e_\beta}} )
   -h'_{\beta'}({\zeta'}^{j'}x^{\frac{1}{e'_{\beta'}}}) ],$$
with
$y=h_{\beta}(\zeta^{j}x^{\frac{1}{e_\beta}}) \in\mathbb C\langle x^*\rangle$ an
equation of the $j$-th probranch of $\mathcal B_\beta$
at $P_0$,
$y=h'_{\beta'}({\zeta'}^{j'}x^{\frac{1}{e'_{\beta'}}}) \in\mathbb C\langle x^*\rangle$
an equation of the $k'$-th probranch of $\mathcal B'_{\beta'}$
at $P_0$, with $\zeta:=e^{\frac {2i\pi}{e_\beta}}$
and $\zeta':=e^{\frac {2i\pi}{e'_{\beta'}}}$.
\end{prop}
With the notations of Proposition \ref{formulemultiplicite},
we get
\begin{equation}\label{branche}
i_{m}(\mathcal C,\mathcal C')=\sum_{\beta=1}^bi_{P_0}(\mathcal B_\beta,V(F')),
\end{equation}
with the usual definition given in \cite{Wall} of intersection number of a branch
with a curve
\begin{equation}\label{branche2}
i_{P_0}(\mathcal B_\beta,V(F'\circ M))=\sum_{j=0}^{e_{\beta}-1}val_x(F'\circ M
   (x,h_{j,\beta}
     (\zeta^{j}x^{\frac{1}{e_\beta}}))).
\end{equation}
\begin{proof}[Proof of Proposition \ref{formulemultiplicite}]
By definition, the intersection number is defined by 
$$i_{m}(\mathcal C,\mathcal C'
)=i_{P_0}(V(F\circ M,F'\circ M)=
\mbox{length}\left(\left(\frac{\mathbb{C}[X,Y,Z]}{(F\circ M,F'\circ M)}\right)_{(X,Y,Z)}\right)$$ 
where $(\frac{\mathbb{C}[X,Y,Z]}{
(F\circ M,F'\circ M)})_{(X,Y,Z)}$ is the local ring in the
maximal ideal $(X,Y,Z)$ of $P_0$ \cite[p. 53]{Hartshorne}.
According to \cite{Fulton}, we have
\[
i_{m}(\mathcal C,\mathcal C ')=\dim_{\mathbb{C}}\left(\left(\frac{\mathbb{C}[X,Y,Z]}{
(F\circ M,F'\circ M)}\right)_{(X,Y,Z) }\right)
\]
Let $f,f'$ be defined by $f(x,y)=F\circ M(x,y,1)$, $f'(x,y)=F^{\prime }
   \circ M(x,y,1)$.
We get
$$i_{m}(\mathcal C,\mathcal C ')=\dim _{\mathbb{C}}\left(\left(\frac{\mathbb{C}[x,y]}{
(f,f')}\right)_{(x,y)}\right)=\dim_{\mathbb C}
   \frac{\mathbb{C}\langle x,y\rangle}{(f,f')}.$$
Recall that, according to the Weierstrass preparation theorem,
there exist two units $U$ and $U'$ of $\mathbb C\langle x,y\rangle$ and
$f_1,...,f_b,f'_1,...,f'_{b'}\in\mathbb C\langle x\rangle[y]$ monic irreducible
such that
$$f=U\prod_{\beta=1}^bf_\beta\ \ \mbox{and}\ \  f'=U'\prod_{\beta'=1}^{b'}f'_{\beta'},$$ 
$f_\beta=0$ being an equation of $\mathcal B_\beta$ and
$f'_{\beta'}=0$ being an equation of $\mathcal B'_{\beta'}$.
According to the Puiseux theorem,
$\mathcal B_\beta$ (resp. $\mathcal B'_{\beta'}$) 
admits a parametrization
\[
\left\{ 
\begin{array}{c}
x=t^{e_{\beta}} \\ 
x=h_{\beta}(t)\in \mathbb{C}\langle t\rangle
\end{array}
\right. \ \mbox{(resp. }\left\{ 
\begin{array}{c}
x=t^{e'_{\beta'}} \\ 
x=h_{\beta'}' (t)\in \mathbb{C}\langle t\rangle
\end{array}
\right. ). 
\]
We know that, for every $\beta\in \{1,..,b\}$ and every
$j\in \{0,..,e_{\beta}\}$,
$h_{\beta}(\zeta ^{j}x^{\frac{1}{e_{\beta}}})\in 
\mathbb{C}\langle x^{\frac{1}{e_{\beta}}}\rangle$
are the $y$-roots of $f_{\beta}$ (resp. $h_{\beta'}({\zeta'} ^{j'}x^{\frac{1}{
e'_{\beta'}}} )\in \mathbb{C}\langle x^{\frac{1}{e'_{\beta'}}}\rangle$
are the $y$-roots of $f'_{\beta'}$).
In particular, we have
$$f_{\beta}(x,y)=\prod_{j=0}^{e_{\beta}-1}(y-h_{\beta}
     (\zeta ^{j}x^{\frac{1}{e_{\beta}}}))
   \ \ \mbox{and}\ \ f'_{\beta'}(x,y)=\prod_{j'=0}^{e'_{\beta'}-1}(y-h'_{\beta'}
     ({\zeta'} ^{j'}x^{\frac{1}{e'_{\beta'}}})).$$
Therefore we have the following
sequence of $\mathbb{C}$-algebra-isomorphisms:
$$\frac{\mathbb{C}\langle x,y\rangle}{(f,f')}=\frac{\mathbb{C}\langle x,y\rangle}{
(\prod_{\beta=1}^{b}f_{\beta}(x,y),f'(x,y))}\cong
\prod_{\beta=1}^{b}A_{\beta},$$
where $A_{\beta}:=\frac{\mathbb{C}\langle x,y\rangle}{
(f_{\beta}(x,y),f'(x,y))}$.
Let $\beta\in\{1,...,b\}$.
We observe that we have
$$A_\beta=\prod_{j=0}^{e_\beta-1}\frac{\mathbb C\langle x\rangle}
  {(f'(x,h_\beta(\xi^jx^{\frac 1 e_\beta})))}.$$
On another hand, we have
\begin{eqnarray*}
D_{\beta}&:=&\frac{\mathbb{C}\langle x^{\frac{1}{e_{\beta}}
  },y\rangle}{(f_{\beta}(x,y),f'(x,y))}
  =\frac{\mathbb{C}\langle x^{\frac{1}{e_{\beta}}},y\rangle}{(\prod
_{j=0}^{e_{\beta}-1}(y-h_{\beta}(\zeta ^{j}x^{\frac{1}{e_{\beta}}})),f'(x,y))}\\
&\cong& \prod_{j=0}^{e_{\beta}-1}\frac{\mathbb{C}\langle x^{\frac{1}{e_{\beta}}},y\rangle
}{(y-h_{\beta}(\zeta ^{j}x^{\frac{1}{e_{\beta}}}),f'(x,y))}\cong
\prod_{j=0}^{e_{\beta}-1} D_{\beta,j}
\end{eqnarray*}
with
$$D_{\beta,j}:=\frac{\mathbb{C}\langle x^{\frac{1}{e_{\beta}}}\rangle }{
(f'(x,h_{\beta}(\zeta ^{j}x^{\frac{1}{e_{\beta}}})))}.$$
We consider now the natural extension of rings $i_\beta:A_{\beta,j}\hookrightarrow 
D_{\beta,j}$ such that
$$\forall g\in A_\beta,\ \ \ val_{x^{1/e_\beta}}((i_\beta(g))(x))=e_\beta val_x(g(x)).$$
We have 
$$D_{\beta}\cong \prod_{j=0}^{e_{\beta}-1}\frac{\mathbb{C}\langle x^{\frac{1}{
e_{\beta}}}\rangle }{(x^{v_{\beta}})},$$ 
where $v_{\beta}$ is the valuation in 
$x^{\frac{1}{e_{\beta}
}}$ of $(f'(x,h_{\beta}(\zeta ^{j}x^{\frac{1}{e_{\beta}}})))$, i.e.
$$v_{\beta}:=val_{t}(f'(t^{e_\beta},h_{\beta}(\zeta ^{j}t ) ))=
   e_\beta\, val_{x}(f'(x,h_{\beta}(\zeta ^{j}x^{\frac 1{e_\beta}} ) )).$$
We get
\begin{eqnarray*}
i_{m}(\mathcal C,\mathcal C ')&=&\sum_{\beta=1}^{b}\dim _{\mathbb{C}
}A_{\beta}=\sum_{\beta=1}^{b}\sum_{j=0}^{e_{\beta}-1}\frac{1}{e_{\beta}}
val_{t}(f'(t^{e_\beta},h_{\beta}(\zeta ^{j}t ) ))
  \\
&=&\sum_{\beta=1}^{b}\sum_{j=0}^{e_{\beta}-1}
val_{x}(f'(x,h_{\beta}(\zeta ^{j}x^{\frac 1{e_\beta}} ) ))
=\sum_{\beta=1}^{b}\sum_{j=0}^{e_{\beta}-1}\sum
_{\beta'=1}^{b'}val_{x}(f'_{\beta'}(x,h_{\beta}({\zeta}^{j}x^{\frac{1}{e_{\beta}}}))).
\end{eqnarray*}
Observe now that
$$val_{x}(f'_{\beta'}(x,h_{\beta}(\zeta ^{j}x^{\frac{1}{
e_{\beta}}})))\in \frac{1}{e_{\beta}}\mathbb{N}$$
and that
\[
f'_{\beta'}(x,h_{\beta}(\zeta ^{j}x^{\frac{1}{e_{\beta}}}))\equiv 
Res(f'_{\beta'},f_{\beta};y)\equiv
\prod_{j'=0}^{e'_{\beta'}-1}(h'_{\beta'}
    ({\zeta'} ^{j'}x^{\frac{1}{e'_{\beta'}}})-h_{\beta}(\zeta ^{j}x^{
\frac{1}{e_{\beta}}}  )), 
\]
where $Res$ denotes the resultant and where $\equiv $ means "up to a non zero scalar".
Finally, we get
\[
i_{m}(\mathcal C,\mathcal C')=
\sum_{\beta=1}^{b}\sum_{j=0}^{e_{\beta}-1}\sum_{\beta'=1}^{b'}
 \sum_{j'=0}^{e'_{\beta'}-1}val_{x}[h'_{\beta'}({\zeta'}^{j'}x^{\frac{1}{e'_{\beta'}
    }} )-h_{\beta}(\zeta ^{j}x^{\frac{1}{e_{\beta}}})]. 
\]
\end{proof}

\end{appendix}\medskip

\noindent 
{\bf Acknowledgements~:\/}

The authors thank Jean Marot for stimulating discussions and for having indicated them
the formula of Brocard and Lemoyne.


\begin{thebibliography}{00}
\bibitem{Brocard-Lemoyne}H.~Brocard, T.~Lemoyne.
Courbes g\'eom\'etriques remarquables. Courbes sp\'eciales planes et gauches. Tome I. 
(French) Nouveau tirage Librairie Scientifique et Technique Albert Blanchard, Paris (1967)
viii+451 pp.
\bibitem{BGG1}J. W. Bruce, P. J. Giblin and C. G. Gibson.
{\em Source genericity of caustics by reflexion in the plane}, 
Quarterly Journal of Mathematics (Oxford) (1982) Vol. 33 (2)
   pp. 169--190.
\bibitem{BGG2}J. W. Bruce, P. J. Giblin and C. G. Gibson.
{\em On caustics by reflexion}, 
Topology (1982) Vol. 21 (2) pp. 179--199.
\bibitem{BG}J. W. Bruce and P. J. Giblin.
{\em Curves and singularities},
Cambridge university press (1984).
\bibitem{Catanese}F.~Catanese. {\em Caustics of plane
curves and their birationality}. Preprint.
\bibitem{CataneseTrifogli}F.~Catanese and C.~Trifogli.
{\em Focal loci of algebraic varieties. I. Special issue in honor of Robin Hartshorne}. 
Comm. Algebra 28 (2000), no. 12, pp. 6017--6057.  
\bibitem{Chasles} M.~Chasles. {\em D\'etermination, par le principe des correspondances,
de la classe de la d\'evelopp\'ee et de la caustique par r\'eflexion d'une courbe g\'eom\'etrique
d'ordre m et de classe n} Nouv. Ann. Math.
2 ser. vol. 10 (1871), p. 97--104, extrait C. R. s\'eances A. S. t. LXII.
\bibitem{coolidge} J.~L.~Coolidge. A treatise on algebraic plane curves, Dover, Phenix edition
(2004).
\bibitem{Dandelin} G.~P.~Dandelin. {\em Notes sur les caustiques par r\'eflexion}, (1822).
\bibitem{Fantechi} B.~Fantechi. {\em The Evolute of a Plane Algebraic Curve}, (1992) UTM 408,
University of Trento.
\bibitem{Fisher}G.~Fischer. {\em Plane algebraic curves}, AMS  2001.
\bibitem{Fulton}W.~Fulton. {\em Intersection Theory}, 2nd ed., Springer, 1998.
\bibitem{Halphen}G.~H.~Halphen. {\em M\'emoire sur les points singuliers des courbes 
alg\'ebriques planes}. Acad\'emie des Sciences t. XXVI (1889) No 2.
\bibitem{Hartshorne}R.~Hartshorne. Algebraic geometry. Graduate Texts in Mathematics, 
No. 52. Springer-Verlag, New York-Heidelberg, (1977).
\bibitem{fredsoaz1}A.~Josse, F. P\`ene. On the degree of caustics by reflection. To appear in Communications in Algebra.
\bibitem{fredsoaz3}A.~Josse, F. P\`ene. {\em Degree and class of caustics by reflection for a generic source}. Preprint.
ArXiv:1301.1846.
\bibitem{Quetelet}L.~A.~J.~Quetelet, {\em \'Enonc\'es de quelques th\'eor\`emes nouveaux
sur les caustiques}.
C. G. Q. (1828) vol 1, p.14, p. 147-149.
\bibitem {Salmon-Cayley}G.~Salmon G, {\em A treatise on higher plane curves: 
Intended as a sequel to a treatise on conic sections}. Elibron classics (1934).
\bibitem{Trifogli} C.~Trifogli. {\em Focal Loci of Algebraic Hypersurfaces: a General Theory},
Geom. Dedicata 70 (1998), pp. 1--26.
\bibitem{Tschirnhausen}E. W. von Tschirnhausen, Acta Erud. nov. 1682.
\bibitem{Wall}C.~T.~C.~Wall. {\em Singular Points of Plane Curves}. 
Cambridge University Press. 2004.
\bibitem{Zariski} O.~Zariski. {\em Le probl\`eme des modules pour les branches planes}. 
Course given at the Centre de Math\'ematiques de l'\'Ecole Polytechnique, Paris, October-November 1973. 
With an appendix by Bernard Teissier. Second edition. Hermann, Paris, (1986) x+212 pp.
\end{thebibliography}
\end{document}